\documentclass[10pt,letterpaper]{article}

\usepackage{color}
\usepackage{url}
\usepackage{graphicx}
\usepackage{comment}
\usepackage{array}
\usepackage{amsmath,amsfonts,amssymb,amsxtra,amsthm}
\usepackage{mathrsfs}
\usepackage{lineno}
\usepackage{verbatim}
\usepackage{xy}
\xyoption{all}
\usepackage{stmaryrd}
\usepackage{appendix}
\usepackage{pinlabel}
\usepackage{hyperref}
\DeclareMathOperator{\subgroups}{Subgroups}

\DeclareMathOperator{\covo}{covol}

\DeclareMathOperator{\height}{height}

\DeclareMathOperator{\triangles}{\triangledown}
\DeclareMathOperator{\tripair}{\boxslash}
\DeclareMathOperator{\spairs}{{\mathrm{s}{\tripair}}}

\DeclareMathOperator{\stab}{Stab}
\DeclareMathOperator{\fix}{Fix}

\DeclareMathOperator{\betti}{b_1}

\DeclareMathOperator{\link}{lk}

\DeclareMathOperator{\rk}{rk}

\def\jsj{\textsc{JSJ}}

\newcommand{\HH}{\mathcal{H}}

\newcommand{\XX}{\mathcal{X}}

\def\sl2c{\ensuremath{{SL}(2,\mathbb{C})}}
\def\t1sl2c{{\mathfrak sl}_2\mathbb{C}}
\def\free{\ensuremath{\mathbb{F}}}
\def\zee{\mathbb{Z}}

\def\into{\hookrightarrow}

\makeatletter
\newcommand{\Doubletwo}[3]{ \left\{ #1 \mid #3\right\} }
\newcommand{\Doubleone}[1]{ \left\{ #1 \right\} }
\newcommand{\Doublethr}[3]{ \left\{ #1 \right\}_{#3} }
\newcommand{\set}[1]{%
\@ifnextchar:{\Doubletwo{#1}}{\@ifnextchar_{\Doublethr{#1}}{\Doubleone{#1}}}%
}
\makeatother

\makeatletter
\newcommand{\grouptwo}[3]{ \langle #1 \mid #3\rangle }
\newcommand{\groupone}[1]{ \langle #1 \rangle }
\newcommand{\group}[1]{%
\@ifnextchar:{\grouptwo{#1}}{\groupone{#1}}%
}
\makeatother

\def\zee{\mathbb{Z}}

\newtheorem{theorem}{Theorem}[section]
\newtheorem*{utheorem}{Theorem}
\newtheorem{lemma}[theorem]{Lemma}

\newtheorem{corollary}[theorem]{Corollary}

\newtheorem{remark}[theorem]{Remark}

\theoremstyle{definition}
\newtheorem{definition}[theorem]{Definition}

\usepackage{setspace}

\newcommand{\mnote}[1]{}
\newcommand{\lmnote}[1]{}
\title{Strong accessibility for finitely presented groups}
\author{Larsen Louder \& Nicholas Touikan}

\date{\today}

\begin{document}

\maketitle
\begin{abstract}
  A hierarchy of a group is a rooted tree of groups obtained by
  iteratively passing to vertex groups of graphs of groups
  decompositions.  We define a (relative) slender JSJ hierarchy for
  (almost) finitely presented groups and show that it is finite,
  provided the group in question doesn't contain any slender subgroups
  with infinite dihedral quotients and satisfies an ascending chain
  condition on certain chains of subgroups of edge groups.

  As a corollary, slender JSJ hierarchies of hyperbolic groups which
  are (virtually) without $2$--torsion and finitely presented
  subgroups of $\mathrm{SL}_{n}(\zee)$ are both finite. %SLnZ
\end{abstract}

\section{Introduction}
\thispagestyle{empty}

A group $G$ is said to be \emph{accessible} over a family of subgroups
$\mathcal{C}$ if there is an upper bound to the size of reduced graphs
of groups decompositions of $G$ with edge groups in $\mathcal{C}$. The
classic theorem is due to Grushko and Neumann: If $G=A*B$ is a
nontrivial decomposition of $G$ as a free product then
$\rk(A)+\rk(B)=\rk(G)$, where $\rk(G)$ is the minimal number of
elements needed to generate $G$. This implies that there is an upper
bound to the size of reduced graphs of groups decompositions of a
given finitely generated group $G$ over trivial edge groups. As a
consequence every finitely generated group $G$ admits a free product
decomposition $G\cong G_1*\dotsc*G_p*\free_q$, where each $G_i$ is
freely indecomposable and $\free_q$ is free.

Finitely generated groups are not accessible over the class of small
subgroups, as Dunwoody and Bestvina \& Feighn have produced
counterexamples (with finite and small edge groups, respectively)
~\cite{dunwoody::inaccessible,bf::counterexample}. Finitely presented
groups, on the other hand, are accessible over the class of small
subgroups~\cite{dunwoody::access,bf::folding}. In particular, any
sequence of reduced refinements of graphs of groups decompositions of
a finitely presented group over small edge groups must terminate. A
similar theorem holds for sequences of minimal graph of groups
decompositions of Coxeter groups~\cite{mihalik}.

Dunwoody's theorem, along with Stallings' theorem on groups with
infinitely many ends, implies that any finitely presented group admits
a graph of groups decomposition with finite or one-ended vertex
groups.  (Which we will call the Grushko-Stallings-Dunwoody, or GSD,
decomposition.)

The slender {\jsj} decomposition of a finitely presented group is the
natural generalization of the GSD decomposition to splittings over
slender subgroups, and it is natural to ask if the process of
iteratively passing to vertex groups of slender {\jsj} decompositions
terminates, or, in other words, if a group is \emph{strongly accessible}.

We know of two non-artificial classes of groups which have hierarchies
that must be finite. A Haken hierarchy of a three-manifold gives a
finite hierarchy in this sense: an incompressible two-sided surface in
a three-manifold corresponds to a splitting of its fundamental group
over the fundamental group of the surface. Sela and, independently,
Kharlampovich, Myasnikov and Remeslennikov have shown that the
hierarchy, or \emph{analysis lattice} in Sela's terminology, of a
limit group obtained by alternatingly passing to vertex groups of the
Grushko or abelian {\jsj} decomposition is
finite~\cite{km::irredII,sela::dgog1}. It should be noted that
finiteness of analysis lattices is used to prove finite presentability
of limit groups, rather than the other way around.

Delzant and Potyagailo claim in~\cite{dp} that finitely presented
groups admit finite hierarchies over elementary families (see
Definition~\ref{def::elementaryfamily}), but unfortunately the proof
of~\cite[Lemma~4.10]{dp} is not correct. See
Subsection~\ref{ex::counterexample}. We believe that any proof which
attempts to assign a complexity to each group in a hierarchy is
unlikely to work.

\section{Definitions and results}

A group is \emph{small} if it doesn't contain a non-abelian free
subgroup. An action of a small group on a tree is either
\emph{elliptic} (fixes a point in the tree) \emph{hyperbolic} (has an
axis and acts by translations) \emph{dihedral} (has an axis and acts
dihedrally) or \emph{parabolic} (fixes a point in the boundary but has
no axis)~\cite[p.~453]{bf::folding}. A group is \emph{slender} if all
its subgroups are finitely generated. An action of a slender group on
a tree is either elliptic or stabilizes an axis.

\begin{definition}[Hierarchy]\label{def::hierarchy}
  A \emph{hierarchy} for a group $H$ is a rooted tree of groups $\HH$,
  with $H$ at the root, such that the descendants of a group $L\in\HH$
  are the vertex groups of a nontrivial graph of groups decomposition
  $\Delta_L$ of $L$. A group $L\in\HH$ is \emph{terminal} if $L$ has no
  descendants. 

  A hierarchy is \emph{slender} if all graphs of groups decompositions
  $\Delta_L$, $L\in\HH$, are over slender edge groups.  A slender
  hierarchy is \emph{hyperbolic} if for all $L\in\HH$, if $E<L$ is an
  edge group of $\Delta_L$, then for every $L' \in \HH$ and every
  conjugate $E^g$ of $E$, the action of $E^g\cap L'$ on
  $T_{\Delta_{L'}}$, the tree associated to $\Delta_{L'}$, is either
  elliptic or hyperbolic, but not dihedral.
\end{definition}%nick changed this a little

Subgroups which are conjugate to elliptic subgroups in every level of
the hierarchy are particularly important.

\begin{definition}[$\HH$--elliptic]
  Denote the groups at level $n$ of the hierarchy by $\HH^n$.  A
  subgroup $V<L\in\mathcal{H}^n$ is \emph{$\HH$--elliptic} if there is
  a chain $L=L_n>L_{n+1}>L_{n+1}>\dotsc$ such that $L_i\in\HH^i$ is a
  vertex group of $\Delta_{L_{i-1}}$ and elements $h_i\in L_i$ such
  that
  \[
  V<L_n\cap L_{n+1}^{h_n}\cap L_{n+2}^{h_{n+1}h_n}\cap L_{n+3}^{h_{n+2}h_{n+1}h_n}\cap\dotsb,
  \]
  where by $L^h$ we mean $hLh^{-1}$.

\end{definition}

We denote the collection of $\HH$--elliptic subgroups by
$\HH^{\infty}.$ Let $\mathcal{EG}^n_{\HH}$ be the collection of
conjugates of edge groups of graphs of groups decompositions $\Delta_L$
as $L$ varies over all groups in $\HH^n$, let
$\mathcal{EG}_{\HH}=\cup_n\mathcal{EG}^n_{\HH}$, and by
$\mathcal{C}^n_{\HH}$ the collection
%N{D}
\[
\left(\subgroups(\mathcal{EG}^n_{\HH}\setminus\HH^{\infty})\right)\cap\HH^{\infty}
\]
and $\mathcal{C}_{\HH}=\cup_n\mathcal{C}^n_{\HH}$. In plain English,
$\mathcal{C}_{\HH}$ is the collection of $\HH$--elliptic subgroups of
non-$\HH$--elliptic edge groups. 

\begin{definition}[Ascending chain condition]
  \label{def::acc}
  We say that $\HH$ satisfies the \emph{ascending chain condition, or
    acc, on $\mathcal{C}_{\HH}$} if every ascending chain
    \[S_n \leq S_{n+1} \leq S_{n+2} \leq \dotsb\]
    $S_i\in\mathcal{C}^i_{\HH}$, %need \leq instead of < if we want stabilization
    stabilizes.
\end{definition}

%%% REFNOTE: This is standard terminology.

Recall that a group is \emph{almost finitely presented} if it acts
freely and cocompactly on a connected simplicial complex $X$ with
$H^1(X,\zee_2)=0$~\cite{dunwoody::access}. We will use a slightly less
restrictive notion of almost finitely presented.

\begin{definition}[$\HH$--almost finitely presented]
  Let $H$ be a finitely generated group and $\mathcal{E}$ a family of
  subgroups of $H$. We say that $H$ is \emph{almost finitely presented
    relative to $\mathcal{E}$} if $H$ acts cocompactly on a connected
  triangular complex $X$ such that $H^1(X,\zee_2)=0$ and cell
  stabilizers in $X$ are either slender or conjugate into
  $\mathcal{E}$.

  Let $H$ be finitely generated and let $\HH$ be a hierarchy of
  $H$. Call $H$ \emph{$\HH$--almost finitely presented} if $H$ is
  almost finitely presented relative to $\HH^\infty$,
  the family of $\HH$--elliptic subgroups.
\end{definition}

\subsection*{Results}

\begin{theorem}[Main theorem]
  \label{thm::maintheorem}
  Let $H$ be a finitely generated group and let $\HH$ be a slender
  hyperbolic hierarchy of $H$ such that $\HH$ satisfies the acc on
  $\mathcal{C}_{\HH}$.  If $H$ is $\HH$--almost finitely presented
  then for some $N$, there is a constant $C$ such that each
  $L\in\HH^m$, $m\geq N$, has a finite hierarchy $\XX_L$ of height at
  most $C$ whose terminal groups are either $\HH$--elliptic or
  slender.
\end{theorem}

We apply the main theorem by imposing various conditions on finitely
presented groups which guarantee that their slender {\jsj} hierarchies
are hyperbolic.  Call a slender group $E$ with an infinite dihedral
($D_\infty=\zee_2*\zee_2$) quotient \emph{$D_\infty$--slender}. If $E$
has no $D_\infty$--slender subgroups it is
\emph{$\zee$--slender}. Then a hierarchy for which every edge group is
$\zee$--slender is hyperbolic. Note that finite groups and Tarski
monsters are $\zee$--slender. Call a graph of groups decomposition
$\zee$--slender if its edge groups are $\zee$--slender.

\begin{corollary}
  \label{cor::nodihedral}
  Let $H$ be finitely presented without any $D_\infty$--slender
  subgroups. Let $\HH$ be the hierarchy such that $\Delta_L$ is a
  slender {\jsj} decomposition (See~\cite{fuji::jsj} and
  \S\S~\ref{sec::remarksonacc} and \ref{sec::jsjhierarchy}.) of $L$,
  and such that if $L\in\HH$ is slender then $L$ is terminal. If $\HH$
  satisfies the acc on $\mathcal{C}_{\HH}$ then $\HH$ is finite.
\end{corollary}

Corollary~\ref{cor::nodihedral} also holds for groups which are almost
finitely presented relative to a family of subgroups $\mathcal{E}$,
where each decomposition in $\HH$ is the relative slender {\jsj}
decomposition.

In Section~\ref{sec::jsjhierarchy} we show that the {\jsj} hierarchy of
a relatively hyperbolic group satisfies the acc on
$\mathcal{C}_{\HH}$. Also note that a relatively hyperbolic group
contains a two ended $D_\infty$--slender group if and only if it
contains a non-central element of order two.

\begin{corollary}[(Relatively) Hyperbolic groups]
  \label{cor::relhyp}
  Suppose $G$ relatively hyperbolic, finitely generated, and without a
  non-central element of order two. Then the hierarchy $\HH$ such that
  $\Delta_L$ is the slender {\jsj} decomposition relative to
  peripheral subgroups which are not two-ended is finite. If $G$ is
  virtually without $2$-torsion (for example, if $G$ is residually
  finite) hyperbolic group then the slender {\jsj} hierarchy of $G$ is
  finite.
 
  If $G$ is toral relatively hyperbolic the same holds but
  for the full abelian {\jsj} decomposition.
\end{corollary}

If $\left[G,G_1\right]<\infty$ and $H<G$ acts nontrivially on some
tree $T$ with slender edge stabilizers then $H\cap G_1$ is finite
index in $H$ and acts on $T$ with slender edge stabilizers;
furthermore, if $T$ corresponds to the (relative) slender
{\jsj} decomposition of $H$ then $T/(H\cap G_1)$ is obtained from the
(relative) slender {\jsj} decomposition of $H\cap G_1$ by possibly
removing valence two slender vertex groups or cutting enclosing vertex
groups of the (relative) slender {\jsj} decomposition of $H\cap G_1$
along essential simple closed curves. In particular, the non-slender
non-enclosing vertex groups of $T/(H\cap G_1)$ are finite index
subgroups of the non-slender non-enclosing vertex groups of the
{\jsj} of $G_1$. Thus if the slender {\jsj} hierarchy of $G$ has a
non-slender non-enclosing vertex group at level $n$ then so does
$G_1$. The corollary then follows for residually finite hyperbolic
groups by observing that any residually finite hyperbolic group is
virtually torsion free; in particular virtually without two-torsion.

%%%%%%%%%%%%%%%
%SLnZ
%%%%%%%%%%%%%%%

Similarly, suppose $G$ is a finitely presented subgroup of
$\mathrm{SL}_n(\zee)$. Since $\mathrm{SL}_n(\zee)$ is virtually
torsion free, so is $G$. The union $S$ of a chain
\[ S_1 \leq S_2 \leq \ldots \]
of slender subgroups of $G$ is virtually solvable by Tit's
alternative, and by \cite[\S2 Corollary~1]{segal2005polycyclic} $S$ is
virtually polycyclic, hence slender. Any slender hierarchy $\HH$ of
$G$ therefore satisfies the acc on $\mathcal{C}_{\HH}$. Since strong
accessibility passes to finite index overgroups we have:

\begin{corollary}
  \label{cor::slnz}
  The slender JSJ hierarchy of any finitely presented subgroup of
  $\mathrm{SL}_n(\zee)$ is finite.
\end{corollary}

%%%%%%%%%%%%%%%
%SLnZ
%%%%%%%%%%%%%%%

\cite[3.2]{dp} holds if we impose the ascending chain condition on
finite subgroups of elements in an elementary family.  (See
Definition~\ref{def::elementaryfamily})

\begin{theorem}[\emph{c.f.} {\cite[3.2]{dp}}]
  \label{dp::variation}
  Let $G$ be finitely presented and let $\mathcal{C}$ be an elementary
  family of subgroups of $G$. Suppose that any ascending chain of
  finite subgroups of elements of $\mathcal{C}$ eventually stabilizes,
  and that two-ended subgroups of $G$ are $\zee$--slender. Then $G$
  has a hierarchy $\HH$ over edge groups in $\mathcal{C}$ such that
  terminal groups of $\HH$ are either in $\mathcal{C}$ or don't split
  over an element of $\mathcal{C}$.
\end{theorem}

Note that the hierarchy $\HH$ in Theorem~\ref{dp::variation} is not a
priori canonical, whereas the non-slender vertex groups appearing in
the slender {\jsj} hierarchy are.

\subsection*{Acknowledgements}

We would like to thank Thomas Delzant and Leonid Potyagailo for some
interesting conversations.  Thanks to Daniel Groves for help with an
earlier version of the paper, as well as to Peter Ha\"issinsky and
Matias Carrasco Piaggio for encouragement, and to Thierry Coulbois for
bringing our attention to~\cite{guirardel::core}. We also thank Lee
Mosher, Misha Kapovich, and Ian Agol for input on finite subgroups of
relatively hyperbolic groups.

The first author was supported in part by EPSRC grant EP/D073626/2 and
NSF grant DMS-0602191. The second author was supported by an NSERC
PDF, the ANR grant ANR-2010-BLAN-116-01 GGAA
and a Fields Postdoctoral Fellowship.

\section{Dunwoody/Delzant-Potyagailo resolution}
\label{sec::dpconstruction}

Given a simplicial complex $X$ with a free $G$--action and a $G$--tree
$T$ there is always a $G$--equivariant map from $X$ to $T$. If $T$ is
simplicial and the map is chosen reasonably well, preimages of
midpoints of edges form a subset of $X$ called a \emph{pattern}, the
connected components of which are two-sided \emph{tracks}. Patterns
were introduced by Dunwoody in~\cite{dunwoody::access} to show that
(almost-) finitely presented groups are accessible, and used
in~\cite{dunwoodyjsj} to construct a {\jsj} decomposition for finitely
presented groups over slender edge groups. 

If the action on $X$ is not free there is in general no
$G$--equivariant map $X\to T$. The construction below is a
generalization of~\cite[4.1-4.9]{dp},
after~\cite{dunwoody::access}.  Like them, we construct a class of
spaces such that if $T$ is a (suitable) $G$--tree then there are
$G$--equivariant maps $X\to \hat T = T\cup\partial T$, where $\partial
T$ is the boundary at infinity of $T$.

\subsection*{$\HH$--complexes}

All complexes in the sequel are at most two-dimensional.

\begin{definition}[$\HH$--complex]\label{def::hhcomplex}
  Let $H$ be a finitely generated group with a hierarchy $\HH$ over a
  class of groups $\mathcal{C}$. For $G<H$, an $\HH$--complex for $G$
  is a connected simplicial complex $X$ with $H^1(X,\zee_2)=0$, $X/G$
  compact, and with cell stabilizers in $\mathcal{C}$ or
  $\HH^{\infty}$. An $\HH$--complex is \emph{nondegenerate} if it
  contains a triangle, and is \emph{degenerate} if it doesn't.

  The number of orbits of triangles in an $\HH$--complex $X$ is
  denoted by $\covo(X)$.
\end{definition}

Denote the stabilizer of a cell $c\subset X$ by $\stab_X(c)$, the
pointwise stabilizer by $\stab^+_X(c)$, and if $K$ acts on a space $Z$
denote the fixed point set of $K$ by $\fix_Z(K)$. If $X$ and $Z$ are
clear from the context we will omit them.

Let $X$ be a triangular CW--complex. (A CW--complex whose two-cells
have at most three sides.) If $X$ is not simplicial, let $Y$ be the
triangular complex whose vertices are the vertices of $X$, whose edges
are determined by (unordered) pairs of distinct vertices which are the
endpoints of some edge in $X$, and whose triangles are determined by
(again unordered) triples of distinct vertices of $X$ which are
contained in a triangle. There is a continuous map $X\to Y$ which maps
cells to cells of equal or lower dimension, and if the dimensions are
the same then it maps interiors of cells homeomorphically to their
images.

The next lemma is obvious.

\begin{lemma}
  If $X$ is simply connected then $Y$ is simply connected. If
  $\mathrm{H}^1(X,\zee_2)=1$ then $\mathrm{H}^1(Y,\zee_2)=1$.
\end{lemma}

If $X\to Y$ is not a homeomorphism then we say that $X$ is
\emph{reducible}, and if it is then $X$ is \emph{reduced}. The space
$Y$ constructed above is said to be obtained from $X$ by
\emph{reducing}. If $X$ is equipped with an (combinatorial) action of
a group $G$ then $Y$ naturally inherits an action of $G$, and if cell
stabilizers in $X$ are in some class $\mathcal{C}$ (for example if $X$
is an $\HH$--complex and $\mathcal{C}$ is the collection of slender or
$\HH$--elliptic subgroups) which is closed under passing to subgroups
and index-two supergroups then cell stabilizers in $Y$ are elements of
$\mathcal{C}$ as well.

\begin{remark}
  It is \emph{not} necessarily the case that $\stab^+(c)=\stab(c)$ for
  a cell $c$ in $Y$, even if this is the case in $X$.
\end{remark}

\subsection*{Cut trees}

A \emph{cutpoint} in a simplicial complex $X$ is a vertex $v$ such
that $X\setminus v$ has more than one component, and a
\emph{cutpoint-free component} of $X$ is a maximal connected
subcomplex which isn't separated by a cutpoint. Any simplicial complex
is a union of cutpoint-free components which are either disjoint or
meet in a single vertex. Suppose $X$ is connected and let
$\{Y_{\alpha}\}$ be the collection of cutpoint-free components in $X$,
let $\{v_{\beta}\}$ be the collection of cutpoints, and finally let
$T_X$ be the tree whose vertex set is the collection of cutpoint-free
components and cutpoints of $X$, and whose edges are given by pairs
$(Y_{\alpha},v_{\beta})$ such that $v_{\beta}\in Y_{\alpha}$.

Suppose that $X$ is connected, $H^1(X,\zee_2)=0$, and that $X$ doesn't
have any cutpoints. A \emph{cut-edge} is an edge $e$ such that
$X\setminus e$ has at least two components. We mimic the definition
above and let the cut-edge tree $S_X$ be the tree whose vertices are
the maximal connected cut-edge-free components $\{Y_{\alpha}\}$ of $X$
and cut-edges $\{e_{\gamma}\}$ in $X$, and whose edges are given by
pairs $(Y_{\alpha},e_{\gamma})$ such that $e_{\gamma}\subset
Y_{\alpha}$.

\subsection*{Resolving actions on trees}
\label{sec::tracks}

Suppose $X$ is an $\HH$--complex for $G<G'\in\HH$ and let $T$ be the
tree associated to $\Delta_{G'}$. Cell stabilizers in $X$ might not
act elliptically in $T$, and there is therefore no $G$--equivariant
map $X\to T$. If $\HH$ is a slender hyperbolic hierarchy then each
cell stabilizer in $X$ fixes a point in $\hat T = T \cup \partial
T$. We exploit this fact to produce a $G$--equivariant map $X\to\hat
T$, a la Dunwoody and Delzant-Potyagailo. %N{R}

\begin{definition}\label{def::resolution}
  Let $X$ be a reduced triangular complex with an action of a group
  $G$, and let $G$ act on a tree $T$ without inversions. A
  $G$--equivariant map $\rho\colon X\to \hat T$ such that vertices are
  mapped to vertices or points in $\partial T$, interiors of edges are
  mapped homeomorphically to interiors of arcs in $T$, and each
  intersection of a triangle $t$ and a connected component of the
  preimage of a midpoint of an edge of $T$ is an embedded closed arc
  connecting distinct edges of $t$ is a \emph{resolution}.
\end{definition}

\begin{lemma}
  \label{lem::resolvinglemma}
  Let $G$ act on a triangular complex $X$, and let $T$ be the tree
  associated to a graph of groups decomposition of $G$. Suppose that
  \begin{itemize}
  \item vertex stabilizers in $X$ act either elliptically,
    hyperbolically, or parabolically on $T$, and 
  \item if $W\subset
    X$ is a connected subset of $X^1$ such that for all cells $e\subset W$
    $\stab_X(e)$ acts hyperbolically or parabolically in $T$. Then all
    stabilizers of vertices in $W$ have a common fixed point in
    $\partial T$.
  \end{itemize}
  Then there is a resolution $\rho\colon X\to \hat T$.
\end{lemma}

\begin{remark}
  Lemma~\ref{lem::resolvinglemma} is sharp, which is one reason our
  proof of strong accessibility only works for splittings over slender
  edge groups or edge groups in an elementary family. These are
  seemingly the only natural hypotheses which guarantee the lemma
  holds.
\end{remark}

\begin{proof}
  Since cell stabilizers in $X$ don't act dihedrally on $T$,
  $\fix_{\hat T}(\stab(c))$
  is not empty for all cells $c$ in $X$.

  Let $v$ be a representative of an orbit of vertices in $X$. Choose
  arbitrarily a point $\rho(v)\in \fix_{\hat T}(\stab(v))\in \hat T$,
  provided that if $\stab(v)$ is elliptic then $\rho(v)\in T^0$, and
  that if $v$ and $w$ are contained in a connected subset $W\subset
  X^1$
  as in the second bullet then $\rho(v)=\rho(w)$ is a point fixed by
  all stabilizers of cells in $W$.  For each $g\cdot v$ in the orbit
  of $v$, set $\rho(g\cdot v)=g\cdot\rho(v)$. Repeat over all orbits
  of vertices.

  Let $e$ be a representative of an orbit of edges of $X$, and let $v$
  and $w$ be the endpoints of $e$. Suppose that $\stab_X(e)$ inverts
  $e$ and acts elliptically in $T$. Let $b_e$ be the fixed point of
  the stabilizer of $e$ and choose $c\in\fix_T(\stab_X(e))$. Let $r_v$
  and $r_w$ be the two arcs/rays in $T$ connecting $c$ to $\rho(v)$
  and $\rho(w)$, respectively, and set $\rho(b_e)$ to be the furthest
  point in the non-empty, but possibly degenerate, arc $r_v\cap r_w$
  from $c$. In particular if $\rho(v)=\rho(w)\in\partial T$ then set
  $\rho(e)=\rho(v)$. Then $\stab(e)$ stabilizes the (possibly
  degenerate) arc in $T$ connecting $\rho(v)$ and $\rho(w)$. Map
  $\left[b_e,v\right]$ homeomorphically to the (possibly degenerate)
  interval connecting $\rho(v)$ and $\rho(b_e)$ in $\hat T$ and extend
  equivariantly. If $\stab(e)$ doesn't invert $e$ then map $e$ in the
  obvious way to the (again, possibly degenerate) arc connecting
  $\rho(v)$ to $\rho(w)$. Repeat over all orbits of edges.

  The edges of a triangle $t$ determine a (possibly degenerate) tripod
  in $T$. Map $t$ to $T$ equivariantly ($t$ may have nontrivial
  stabilizer), and extend equivariantly to all translates of $t$. See
  Figure~\ref{fig::extendtotriangles}. Repeat over all orbits of
  triangles. (As with edges which are inverted, some care must be
  taken since $\stab(t)$ might not fix $t$.)
\end{proof}

\begin{figure}[ht]
  \centerline{\includegraphics[width=\textwidth]{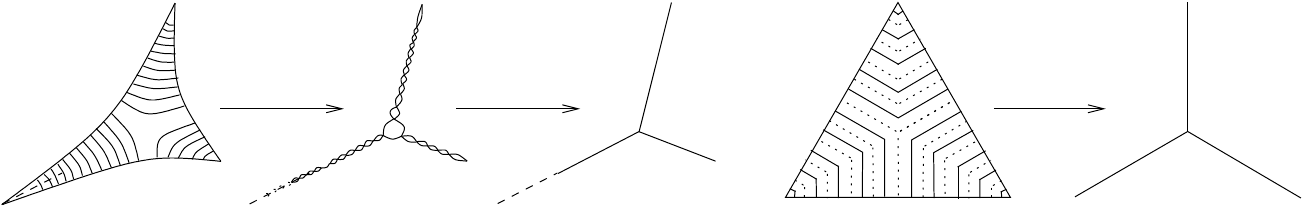}}
  \caption{The first sequence represents the map on a typical
    triangle. The vertex in the lower left corner is sent to
    $\partial T$, which is indicated by the dotted lines. The center
    picture represents the triangle after crushing preimages of
    midpoints of edges, which introduces bigons and creates a new
    triangle. This is essentially~\cite[Dessins 1 et 2]{dp}. The
    second map illustrates a case where
    $\stab(t)\neq\stab^+(t)$. Preimages of vertices are represented by
    dotted lines and preimages of midpoints of edges are solid
    lines. If the stabilizer of $t$ acts as $\mathrm{S}_3$ on $t$ then
    $\rho$ must send a tripod connecting the centers of edges of $t$
    to a point in $T$.}
\label{fig::extendtotriangles}
\end{figure}

\begin{corollary}[{c.f.~\cite[\S4.1]{dp}}]
  \label{cor::resolvinglemma}
  Given a slender hyperbolic hierarchy $\HH$ of a group $H$, an
  $\HH$--complex $X$ for $G<H$, and a slender $G$--tree $T$ from
  $\HH$, there is a resolution $\rho\colon X\to \hat T$.

  If $\mathcal{C}$ is an elementary family in a group $G$, $X$ a
  reduced $G$--complex with cell stabilizers in $\mathcal{C}$, and $T$
  is a $G$--tree with edge stabilizers in $\mathcal{C}$ such that no
  element of $\mathcal{C}$ acts dihedrally on $T$, then there is a
  resolution $\rho\colon X\to \hat T$.
\end{corollary}

Let $G$, $X$, and $T$ be as in the first paragraph of
Corollary~\ref{cor::resolvinglemma}. We divide our treatment of the
map $\rho$ constructed in Lemma~\ref{lem::resolvinglemma} into two
cases.

\subsubsection*{Type 1: $\rho^{-1}(\partial T)$ doesn't contain an edge of $X$.}

Let $X^*=X\setminus\rho^{-1}(\partial T)$. For each edge $e$ of $T$
let $m_e$ be the midpoint of $e$, and let $\Lambda'$ be the
one-complex $\rho^{-1}(\cup_e\{m_e\})\subset X^*$. Call a connected
component $\lambda$ of $\Lambda'$ \emph{essential} if both components
of $X\setminus \lambda$ are unbounded and $\lambda$ is not parallel to
the link of a vertex, and let $\Lambda$ be the union of all essential
components of $\Lambda'$. For the remainder of the paper $\Lambda'$
and $\Lambda$ will be used to indicate patterns constructed in the
manner described above.

Let $X^*/\Lambda$ be the space obtained by collapsing each connected
component of $\Lambda$ to a point, and let $X_T$ be the space obtained
by reducing. (See~\cite[4.2]{dp}.) The stabilizers of the vertices
corresponding to connected components of $\Lambda$ are slender, and
there is a $G$--equivariant map $X^*/\Lambda\to X_T$. A simple case of
this procedure is illustrated in Figure~\ref{fig::counterexample}.

\begin{lemma}[{\emph{c.f.} \cite[Lemma~4.9]{dp}}]
  \label{lem::simplyconnected}
  If $\pi_1(X)=1$ then $\pi_1(X^*/\Lambda)=1$. If $H^1(X,\zee_2)=0$
  then $H^1(X^*/\Lambda,\zee_2)=0$.
\end{lemma}

See Figure~\ref{fig::octagon}.

\begin{proof}
  Let $Z$ be a connected component of $X^*$, let $Y$ be its closure in
  $X$, and let $W$ be the connected component of $X^*/\Lambda$
  corresponding to $Z$. Let $B$ be the second barycentric subdivision
  of $Y$, $C$ the union of simplices in $B$ which miss
  $\rho^{-1}(\partial T)$, $A$ the union of simplices in $B$ which
  meet $\rho^{-1}(\partial T)$, and let $L=A\cap C$.

  Consider the Mayer-Vietoris sequence for the pair of subspaces $A$
  and $C$.
  \[
  \dotsb\to H^1(Y,\zee_2)\to H^1(A,\zee_2)\oplus H^1(C,\zee_2)\to H^1(L,\zee_2)\to\dotsb
  \]
  Each connected component of $A$ is contractible, the inclusion
  $C\into Z$ is a homotopy equivalence, and since $H^1(Y,\zee_2)=0$,
  there is an exact sequence
  \[
  0\to H^1(Z,\zee_2)\to H^1(L,\zee_2)
  \]
  It therefore suffices to show that any closed path in $L$ dies under
  the map $Z\to W$. Similarly for the fundamental group: all $\pi_1$
  is carried by $L$.

  Let $d$ be a reduced edge path in $L$. Then there is a vertex a
  vertex $v\in Y\cap \rho^{-1}(\partial T)$ such that $d$ is homotopic
  in the star of $v$ to an edge path $e_0\dotsb e_{n-1}$ in the link
  of $v$. Let $t_0,\dotsc,t_{n-1}$ be the triangles in $Y$ with
  $e_i,\{v\}\subset t_i$, and let $f_0,\dotsc,f_{n-1}$ be the edges
  connecting $v$ to $e_i$ such that the boundary of $t_i$ is formed by
  $e_i$, $f_i$, and $f_{i+1}$. Since $\rho(e_i)\neq \rho(v)$ there is
  an edge $g$ of $T$ such that $\rho^{-1}(m_e)\cap t_i$ is a single
  arc $a_i$ connecting $f_i$ to $f_{i+1}$ in $t_i$ for each $i$. This
  implies the collection of arcs $\{a_i\}$ forms a closed loop, and
  $d$ is homotopic \emph{in $Z$} to the path $a_0\dotsb a_{n-1}$ in
  $\Lambda\subset Z$. Thus $d$ has nullhomotopic image in $W$. Hence
  if $X$ is acyclic then $W$ is as well, and if $X$ has trivial
  fundamental group then so does $W$.
\end{proof}

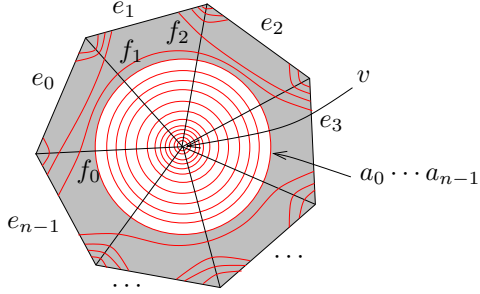
\begin{figure}
\labellist
\pinlabel $e_0$ [br] at 31 224
\pinlabel $e_1$ [br] at 124 307
\pinlabel $e_2$ [bl] at 254 287
\pinlabel $e_3$ [bl] at 322 175
\pinlabel $e_{n-1}$ [tr] at 37 89
\pinlabel $f_0$ [t] at 63 158
\pinlabel $f_1$ [bl] at 91 254
\pinlabel $f_2$ [br] at 186 276
\pinlabel $a_0\dotsb a_{n-1}$ [l] at 366 128
\pinlabel $\dotsb$ [tl] at 269 50
\pinlabel $\dotsb$ [tr] at 136 16
\pinlabel $v$ [bl] at 365 232
\endlabellist
  \centerline{\includegraphics[width=.35\textwidth]{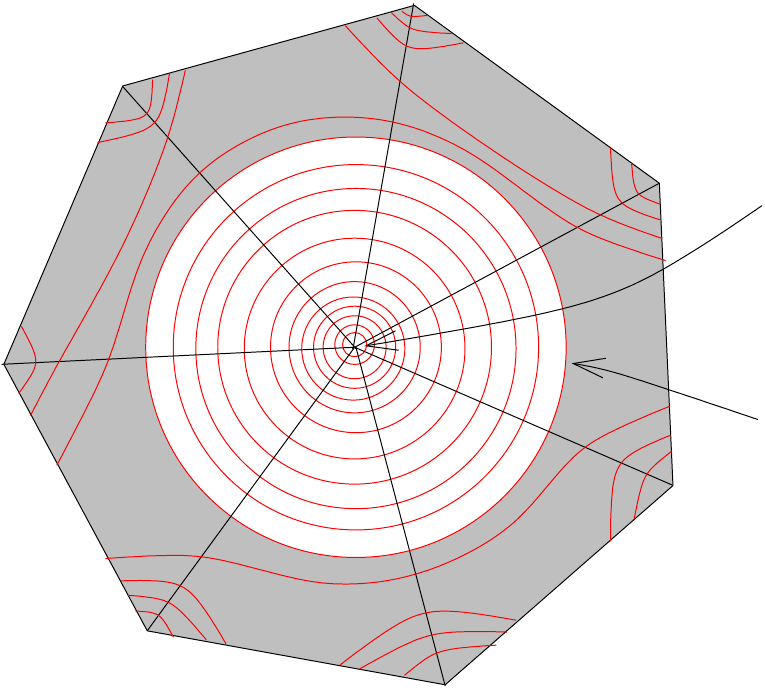}}
  \caption{Illustration for Lemma~\ref{lem::simplyconnected}. The
    outer loop represents the path $e_0\dotsb e_{n-1}$, which is
    homotopic (with the homotopy represented by the shaded annulus) to
    the path $a_0\dotsb a_{n-1}$, which is a path in $\Lambda$, hence $d$
    is trivial in $X^*/\Lambda$.}
\label{fig::octagon}
\end{figure}

\subsubsection*{Type 2: $\rho^{-1}(\partial T)$ contains an edge of
  $X$.}

Let $X_T$ be the complex obtained by collapsing each connected
component of $\rho^{-1}(\partial T)$ to a point and reducing. 

% Suppose
% that $\rho^{-1}(\partial T)$ contains one-cell.

\begin{lemma}
  \label{lem::collapsesmall}
  Vertex stabilizers in $X_T$ are either 
  \begin{itemize}
  \item vertex stabilizers from $X,$
  \item act hyperbolically in $T$, and are HNN extensions of subgroups
    of edge groups in $T/G$, or
  \item act parabolically in $T$, and are strictly ascending HNN
    extensions of subgroups of edge groups in $T/G$.
  \end{itemize}
  Each edge stabilizer in $X_T$ is, or has an index two subgroup which
  is, either a subgroup of a conjugate of an edge group in $T/G$.

  If $X$ is an $\HH$--complex for $G<L\in\HH$ and $T=T_{\Delta_L}$
  then vertex stabilizers in the third category are small but not
  slender. If $X$ is a $G$--complex with stabilizers in an elementary
  family $\mathcal{C}$ then stabilizers in $X_T$ are in $\mathcal{C}$.
\end{lemma}

Let $\HH$ be a slender hierarchy and 
suppose $X$, $G$ and $T$ are as above. Since $X_T$ potentially has
small but not slender vertex stabilizers, i.e., $\rho\colon X\to \hat
T$ is type 2, it is \emph{not}, in general, an $\HH$--complex.

\begin{proof}
  Let $v$ be a vertex in $X_T$. Denote by $\stab(v)$ by $G_v$. If $v$
  is a vertex from $X$ then clearly $G_v$ is a vertex stabilizer from
  $X$. Suppose $v$ corresponds to a connected component $V\subset X$
  %N{It's good to know where $V$ lives}
  of $\rho^{-1}(p)$, with $p\in\partial T$. Clearly $\stab(V)=G_v$. If
  $gV\cap V\neq\emptyset$ then $g\cdot p=p$ and therefore $gV=V$,
  hence $V/G_v\to X/G$ is an embedding and $G_v$ fixes $p$.

  Suppose $X$ is an $\HH$--complex. Since $X/G$ is compact, $G_v$ acts
  cocompactly on $V$, and since $V$ has slender cell stabilizers,
  $G_v$ is finitely generated. Let $T'\subset T$ be the union of axes
  of elements of $G_v$, and consider the quotient $T'/G_v$. Since
  $G_v$ fixes an end in $T$, it fixes an end in $T'$, and $T'/G_v$ is
  therefore an ascending HNN extension with slender edge groups, hence
  is either small or slender, and if small, it acts parabolically on
  $T$.

  Let $e$ be an edge in $X_T$. The stabilizer of $e$ either fixes the
  endpoints $v$ and $w$ of $e$ or has an index two subgroup which
  does. Let $V$ and $W$ be the preimages of $v$ and $w$ in $X$. Then
  $\stab^+(e)$ stabilizes $V$ and $W$. If $V$ and $W$ are connected
  components of preimages of points in $\partial T$ then $\stab^+(e)$
  fixes a pair of distinct points in $\partial T$, hence is a subgroup
  of an edge group of $T/G$. If $V$ is not a connected component of a
  preimage of any point in $\partial T$ then $V$ is a vertex, and
  $\stab^+(e)$ stabilizes a half-line in $T$, hence is (conjugate
  into) a subgroup of an edge group of $T/G$ in this case as well.
\end{proof}

\section{Remarks on accessibility}
\label{sec::remarksonacc}

Kneser finiteness, existence of a Haken hierarchy, and
Dunwoody/Bestvina-Feighn accessibility all rely on uniform upper
bounds to the number of disjoint non-parallel tracks in two-complexes.

\begin{theorem}
  \label{thm::kneser}
  Let $Y$ be a finite two-dimensional simplicial complex. There is a
  constant $C=C(Y)$ such that if $\Lambda\subset Y$ is a pattern with
  at least $C$ connected components, then two connected components of
  $\Lambda$ are parallel.
\end{theorem}

Bestvina and Feighn's accessibility theorem for finitely presented
groups is used to show that (almost) finitely presented groups have
slender {\jsj} decompositions.

\begin{utheorem}[{\cite[Main theorem]{bf::folding}}]
  \label{bf::accessibility}
  Let $G$ be a finitely presented group. Then there exists an integer
  $\gamma(G)$ such that the following holds:

  If $T$ is a reduced $G$--tree with small edge stabilizers, then the
  number of vertices in $T/G$ is bounded by $\gamma(G)$.
\end{utheorem}

They remark that this holds for almost finitely presented groups, and
that the proof goes through without change. In fact, slightly more is
true:

\begin{itemize}
\item Let $G$ be finitely generated and let $\mathcal{E}$ be a
  collection of subgroups of $G$. The conclusion holds if $G$ acts
  cocompactly on a simplicial complex $X$ with
  $\mathrm{H}^1(X,\zee_2)$ and cell stabilizers which are either
  slender, ascending HNN extensions of slender groups, or are
  conjugate into $\mathcal{E}$, provided that elements of
  $\mathcal{E}$ act elliptically in $T$.
\item Similarly, if $G$ has a finite hierarchy $\mathcal{X}$ over edge
  groups which are slender, are ascending HNN extensions of slender
  subgroups, and such that each terminal leaf of $\mathcal{X}$ is
  either slender, an ascending HNN extension of a slender groups, or
  of the form in the previous bullet, then the conclusion holds,
  provided that elements of $\mathcal{E}$ act elliptically in $T$.
\end{itemize}

Accessibility of (almost) (relatively) finitely presented groups
ensures the existence of a (relative) {\jsj} decomposition. In
\S\ref{sec::jsjhierarchy} we will use the above to define a (relative)
{\jsj} hierarchy of (almost) (relatively) finitely presented
groups. Since slender subgroups of finitely presented groups are not
necessarily finitely presented we must work in the category of almost
finitely presented groups.

\section{Counterexample to the proof of {\cite{dp}}}
\label{ex::counterexample}

This section illustrates some of the problems with the approach to
strong accessibility taken by~\cite{dp}. We sketch their proof below,
and try to make clear why such an approach is unlikely to be
successful.

\begin{definition}[{\cite[\S1.1]{dp}}]
  \label{def::elementaryfamily}
  An \emph{elementary family} in a group $G$ is a family $\mathcal{C}$
  of subgroups which is
  \begin{itemize}
    \item Closed under conjugation and passing to infinite subgroups.
    \item Each infinite subgroup of $\mathcal{C}$ is contained in a
      unique maximal element of $\mathcal{C}$, and each ascending
      union of finite elements of $\mathcal{C}$ is an element of
      $\mathcal{C}$.
    \item Elements of $\mathcal{C}$ are \emph{small}, i.e., if
      $A\in\mathcal{C}$ acts minimally on an infinite tree $T$ then
      either $A$ fixes a point in $\partial T$ or stabilizes a pair of
      distinct points in $\partial T$. (Equivalently, no element of
      $\mathcal{C}$ contains a free subgroup.)
    \item If $C\in\mathcal{C}$ is infinite, maximal in $\mathcal{C}$,
      and $C^g=C$, then $g\in C$. In particular, for a maximal $C$ and
      $C'<C$, the normalizer of $C'$ is contained in $C$.
    \end{itemize}
\end{definition}

Elementary families are designed to mimic the family of elementary
subgroups of a (relatively) hyperbolic group, i.e. the class of
virtually cyclic (or peripheral, in the relative case) subgroups.

Delzant and Potyagailo claim:

\begin{theorem}[{\cite[3.2]{dp}}]
  \label{dptheorem}
  Let $G$ be finitely presented and let $\mathcal{C}$ be an elementary
  family of subgroups of $G$. Then $G$ has a hierarchy $\HH$ over edge
  groups in $\mathcal{C}$ such that terminal groups of $\HH$ are
  either in $\mathcal{C}$ or don't split over an element of
  $\mathcal{C}$.
\end{theorem}

Note that we are only able to prove this theorem (\ref{dp::variation})
with the additional hypotheses that the collection of finite subgroups
of elements of $\mathcal{C}$ satisfies the ascending chain condition.

Let $G$ be finitely presented, and suppose $G$ acts simplicially,
cocompactly, and without inversions on a simply connected triangular
complex $X$ with cell stabilizers in an elementary family
$\mathcal{C}$. The quotient $X/G$ is then a complex of groups. The
\emph{T--invariant} of $G$ is the ordered pair
\[
T(G)=\min\{(\vert X/G\vert,\betti(X/G))\mid X\mbox{ as above.} \}
\]
where $\vert X/G\vert$ is the number of triangles in $X/G$. The set of such
ordered pairs is ordered lexicographically.

Let $G$ act on $X$ which achieves its T--invariant. Let $T$ be a
$G$--tree with edge stabilizers in $\mathcal{C}$, and let
$\varphi\colon X\to \hat T$ be the map constructed above. Suppose that
$X^*$ is not connected. Then $G$ splits as a graph of groups
over edge groups in $\mathcal{C}$ so that vertex groups have strictly
lower T--invariant, therefore we may assume that $X^*$ is
connected. They first construct $X_T$ and its cutpoint tree
$T_{X_T}$. Edge stabilizers in $T_{X_T}$ are elements of $\mathcal{C}$
and the quotient graph of groups decomposition has vertex groups
$G_i$, each of which acts as above on a cutpoint free component $X_i$
of $X_T$.

\begin{figure}[ht]
\labellist
\pinlabel $X^*$ [r] at 1 195
\pinlabel $X^*/\Lambda$ [r] at 1 107
\pinlabel $X_T$ [r] at 1 36
\pinlabel $X/\zee$ [r] at 477 195
\pinlabel $(X^*/\Lambda)/\zee$ [r] at 477 107
\pinlabel $X_T/\zee$ [r] at 477 36
\pinlabel $\Lambda$ at 233 258
\pinlabel $\not\exists$ [l] at 577 117
\endlabellist
\centerline{\includegraphics[width=.80\textwidth]{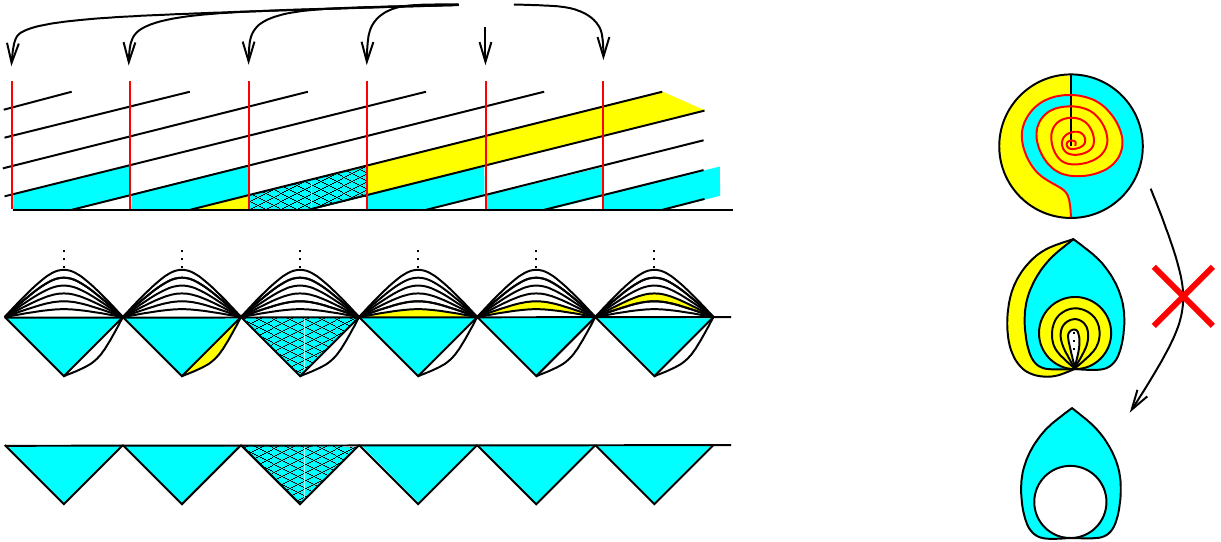}}
\caption{\cite[Lemma~4.10]{dp} violates the no retraction theorem. }
\label{fig::counterexample}
\end{figure}
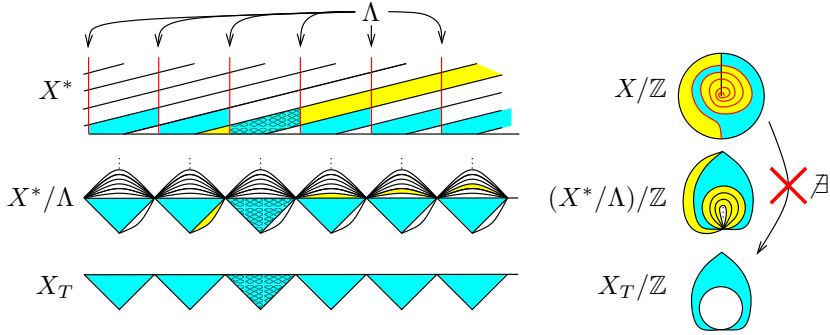

In order to conclude that the $G_i$ have lower T--invariants than $G$
they claim erroneously in~\cite[Lemma~4.10]{dp} that there is a map
$X/G\to X_T/G$ with connected fibers, inducing an isomorphism on
fundamental groups, hence that
\begin{equation}
  \label{eq::wrong}
  \betti(X/G)\geq\sum_i\betti(X_i/G_i)+\betti(T_{X_T}/G),
\end{equation}
where each $X_i$ is a representative of an orbit of cutpoint free
components of $X_T$ and $G_i$ is its stabilizer. They then argue that
if $\betti(T_{X_T}/G)=0$, then there is more than one orbit of cutpoint
free components in $X_T$. In particular, $\vert X/G\vert >\vert
X_i/G_i\vert$, and if $\betti(T_{X_T}/G)>0$ and there is only one
orbit of cutpoint free components then $\vert X/G\vert\geq\vert
X_1/G_1\vert$ and $\betti(X/G)>\betti(X_1/G_1)$, hence $T(G)>T(G_1)$.

The argument used to prove~(\ref{eq::wrong}) is incorrect. See
Figure~\ref{fig::counterexample}. Consider a disk $X/\zee$ with one
orbifold point, labeled $\zee$, and two edges, such that the boundary
of the disk defines a generator. Then $X$ is the (orbihedral)
universal cover of the disk. The cyclic group $\zee$ acts on the line
$T$ and $\pi_1(X_T/\zee)=\zee$. Any continuous map from a disk to a
circle, however, has nullhomotopic image, hence there is in general no
$G$--equivariant map $X\to X_T$.

It is important to note that this is only a counterexample to the
proof of \cite[Lemma~4.10]{dp}, not its conclusion: we know a priori
that the disk $X/\zee$ doesn't achieve the T--invariant of
$\zee$. Their proof however, never actually uses the hypothesis that
$X$ achieves $T(G)$. Any such proof \emph{must} either show that
(\ref{eq::wrong}) holds or that $\vert X/G \vert$ is not minimal. We think it's
unlikely that a proof of strong accessibility along these lines
exists.

\section{Products of trees}
\label{sec::productsoftrees}

Let $G$ be a group and let $T$ and $T'$ be a pair of $G$--trees with
$T/G=\Delta$ and $T'/G=\Omega$. Then $G$ acts diagonally on the
product $T\times T'$. If $S\subset T\times T'$ is a simply connected
$G$--invariant subcomplex, the quotient $S/G$ is a square complex,
which, after~\cite{fuji::jsj}, should be thought of as a complex
of groups. Denote the projections $T\times T'\to T$ and $T\times T'\to
T'$ by $\pi_T$ and $\pi_{T'}$, respectively.

Let $S\subset T\times T'$ be a simply connected $G$--invariant
subset, and suppose that point preimages under $\pi_T$ are
connected. For $v$ a vertex of $\Delta$ let $\widetilde v$ be a lift
of $v$ to $T$. Then $G_v\cong\stab(\widetilde v)$ acts on the tree
$\pi_{T}^{-1}(\widetilde v)$ and $\pi_T^{-1}(\widetilde v)/G_v$ is a
graph of groups decomposition of $G_v$. Similarly, if $m_e$ is a
midpoint of an edge of $\Delta$ then $\pi_T^{-1}(\widetilde m_e)/G_e$
is a graph of groups decomposition of $G_e$.

\begin{theorem}[{\emph{c.f.}~\cite[Th\'eor\`eme principal,
    Corollaire~8.2]{guirardel::core}}]
  \label{lem::core} 
  There is a connected simply connected $G$--equivariant square
  complex $S\subset T\times T'$ of minimal covolume such that the
  projections $S\to T'$ and $S\to T$ have connected point preimages.

  Moreover, if $G$ and all vertex and edge stabilizers are finitely
  generated and the $G$-trees $T, T'$ are cocompact %N{AK}
  then $S/G$ may be taken to be compact.
\end{theorem}

Recall that a $G$--tree is \emph{minimal} if it has no proper
invariant subtrees, and that if a $G$--tree doesn't have a global
fixed point (elliptic) or fixed end (parabolic) then there is a unique
minimal invariant subtree. Though it is customary to assume that all
$G$--trees are minimal, it is necessary to relax this restriction.

Let $S_{\Delta,\Omega}=S/G$, and denote the projections
$S_{\Delta,\Omega}\to\Delta,\Omega$ by $\pi_{\Delta}$ and
$\pi_{\Omega}$. Then $S_{\Delta,\Omega}$ is finite, and if $v$ is a
vertex in $\Delta$, then $\pi^{-1}_{\Delta}(v)$ is a graph of groups
decomposition of $G_v$ corresponding to its action on
$\pi_{T}^{-1}(\widetilde v)$. Similarly, if $m$ is a midpoint of an
edge of $\Delta$ then $\pi^{-1}_{\Delta}(m)$ is a graph of groups
decomposition of $G_e$. Likewise for vertex and edge groups of
$\Omega$. See Figure~\ref{fig::squarecomplex}.

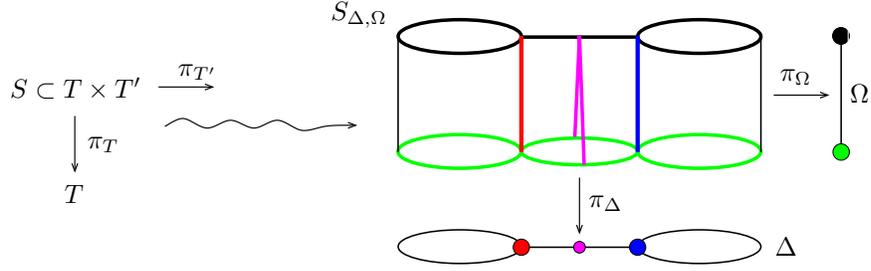
\begin{figure}
\labellist
\pinlabel $\Delta$ [l] at 1917 74
\pinlabel $\pi_{\Delta}$ [l] at 1461 179
\pinlabel $\Omega$ [l] at 2100 447
\pinlabel $\pi_{\Omega}$ [b] at 1987 451
\pinlabel $S\subset T\times T'$ at 225 460
\pinlabel $T$ at 221 202
\pinlabel $T'$ at 639 4602
\pinlabel $\pi_T$ [l] at 234 323
\pinlabel $\pi_{T'}$ [b] at 523 463
\pinlabel $S_{\Delta,\Omega}$ [br] at 1001 595
\endlabellist
\centerline{\includegraphics[width=.95\textwidth]{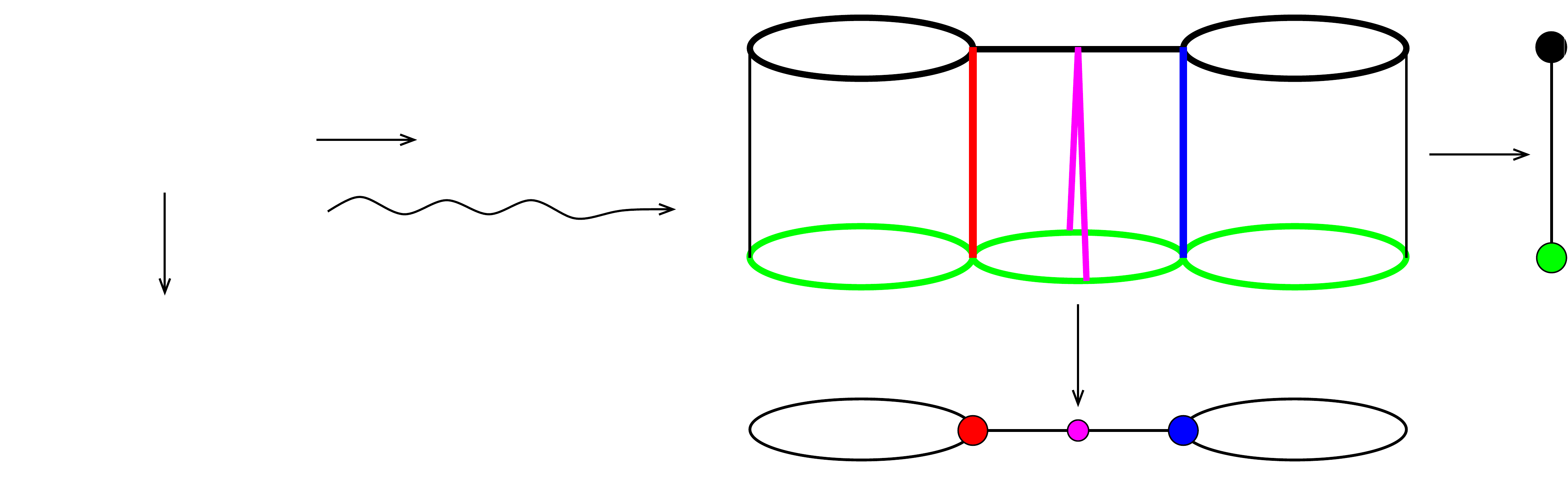}}
\caption{Projections of $S_{\Delta,\Omega}$ to $\Delta$ and $\Omega$.}
\label{fig::squarecomplex}
\end{figure}

The complex $S_{\Delta,\Omega}$ should be thought of as a complex of
groups which interpolates between $\Delta$ and $\Omega$, and is used
extensively in~\cite{fuji::jsj} in the construction of the slender {\jsj} 
decomposition of a finitely presented group.

\begin{lemma}
  \label{lem::newhstructure}
  Let $G$ be a finitely generated group, let $\mathcal{Y}_G$ be a
  finite hierarchy of $G$ over finitely generated edge groups, and let
  $\Delta_G$ be a graph of groups decomposition of $G$ with finitely
  generated edge groups. Then for each vertex group $G_v$ of
  $\Delta_G$ there is a finite hierarchy $\mathcal{X}_{G_v}$, of the
  same height as $\mathcal{Y}_G$,such that vertex and edge groups at
  level $n$ in $\mathcal{X}_{G_v}$ are subgroups of vertex and edge
  groups at level $n$ of $\mathcal{Y}_G$.
\end{lemma}

\begin{proof}
  Let $\Omega_L$ be the decomposition of $L$ for
  $L\in\mathcal{Y}_G$. For each vertex group $L'$ of $\Omega_L$ define
  inductively
  $\Delta_{L'}=\pi_{\Omega}^{-1}(L')\subset\mathcal{S}_{\Delta_L,\Omega_L}$. See
  Figure~\ref{fig::findnewhstructure}.  
  Consider a (nonterminal) vertex group $L'$ of $\Omega_L$ and the
  projection
  \[S_{\Delta_{L'},\Omega_{L'}}\to
  \Delta_{L'}=\pi_{\Omega}^{-1}(L')\subset S_{\Delta_L,\Omega_L}\]
  There is then a natural map $\Pi$ from this hierarchy of square
  complexes to $\Delta$, and if $G_v$ is a vertex group of $\Delta$
  then $\mathcal{X}_{G_v}=\Pi^{-1}(v)$ is a hierarchy of $G_v$ with
  the desired properties.
\end{proof}

\begin{figure}[ht]
\labellist
\pinlabel $\Delta$ [br] at 13 101
\pinlabel $\Omega_L$ [t] at 42 58
\pinlabel $S_{\Delta_L,\Omega_L}$ [b] at 100 116
\pinlabel $\Omega_{L'}$ [t] at 145 29
\pinlabel $S_{\Delta_{L'},\Omega_{L'}}$ [b] at 200 89
\pinlabel $\Omega_{L''}$ [t] at 244 1
\pinlabel $S_{\Delta_{L''},\Omega_{L''}}$ [b] at 301 59
\pinlabel $G_v$ [t] at 1 83
\pinlabel $\Pi^{-1}(G_v)$ [tl] at 313 25
\pinlabel $L'$ [tl] at 76 55
\pinlabel $\Delta_{L'}$ [l] at 139 118
\endlabellist
\centerline{\includegraphics[width=.8\textwidth]{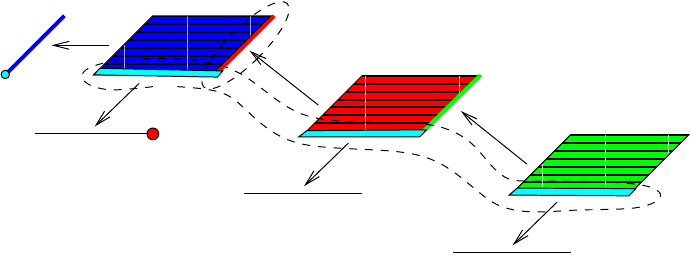}}
\caption{A piece of the hierarchy of square complexes associated to
  $\mathcal{Y}_G$ and $\Delta$.}
\label{fig::findnewhstructure}
\end{figure}
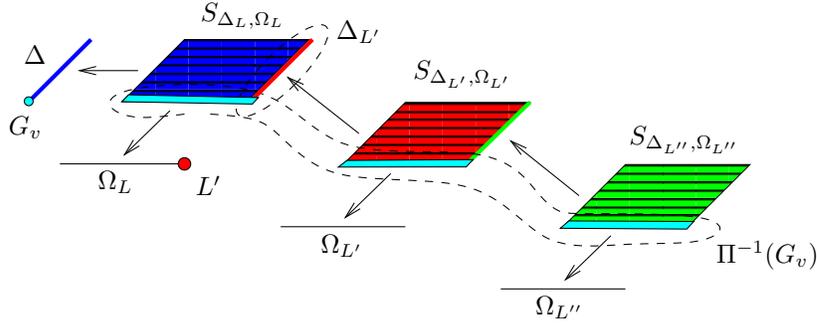

\section{$\HH$--structures}
\label{sec::hstructure}

There is no natural way to construct an $\HH$--complex for each group
$L$ in a hierarchy $\HH$ without losing control over the number of
orbits of triangles. To get around this difficulty we define an
$\HH$--structure, which is a combination of a hierarchy $\mathcal{X}$
(distinct from $\HH$!) and a collection of %N{AL}
$\HH$--complexes (Recall Definition \ref{def::hhcomplex}.) for
terminal groups in $\mathcal{X}.$ We associate, to each group $L$ in a
slender hyperbolic hierarchy $\HH$, an $\HH$--structure $\mathcal{X}_L$
and show in \S\ref{sec::stabilizing} that for groups sufficiently far
down the hierarchy the $\HH$--structures may be taken to have terminal
vertex groups which are $\HH$--elliptic or slender. This will complete
the proof of Theorem~\ref{thm::maintheorem}.

\begin{definition}
  Let $\HH$ be a hierarchy of a group $H$. An \emph{$\HH$--structure}
  on a group $L<H$ is a finite hierarchy over slender or small edge
  groups equipped with an action, for each terminal group $V$ of
  $\mathcal{X}_L$, of $V$ on an $\HH$--complex $X_V$. If $X_V$ is not
  a point then $V$ is \emph{nondegenerate}, and if $X_V$ is a point
  then $V$ is \emph{degenerate}. The complexity $\covo(\mathcal{X}_L)$
  is the total number of orbits of triangles over all $X_V$ under
  their respective actions.

  An $\HH$--structure for $L$ with slender edge groups will be denoted
  by $\mathcal{X}_L$, and if an $\HH$--structure for $L$ possibly has
  small edge groups then it is denoted by $\mathcal{Y}_L$. If
  $L\in\HH$ has an $\HH$--structure $\mathcal{Y}_L$ then we require
  that all non-slender small edge groups in $\mathcal{Y}_L$ act
  parabolically in the Bass-Serre tree $T_{\Delta_L}$.

  The \emph{height} of an $\HH$--structure on $L$ is the number of
  levels in $\mathcal{X}_L$, and is denoted by
  $\height(\mathcal{X}_L)$. We denote graphs of groups decompositions
  in $\mathcal{H}$--structures by $\Omega$, i.e., if
  $L'\in\mathcal{X}_L$ then the graph of groups decomposition of $L'$
  will be denoted by $\Omega_{L'}$.
\end{definition}

\begin{figure}
  \labellist
  \pinlabel $V\curvearrowright X_V$ at 299 259
  \endlabellist
  \centerline{\includegraphics[width=.5\textwidth]{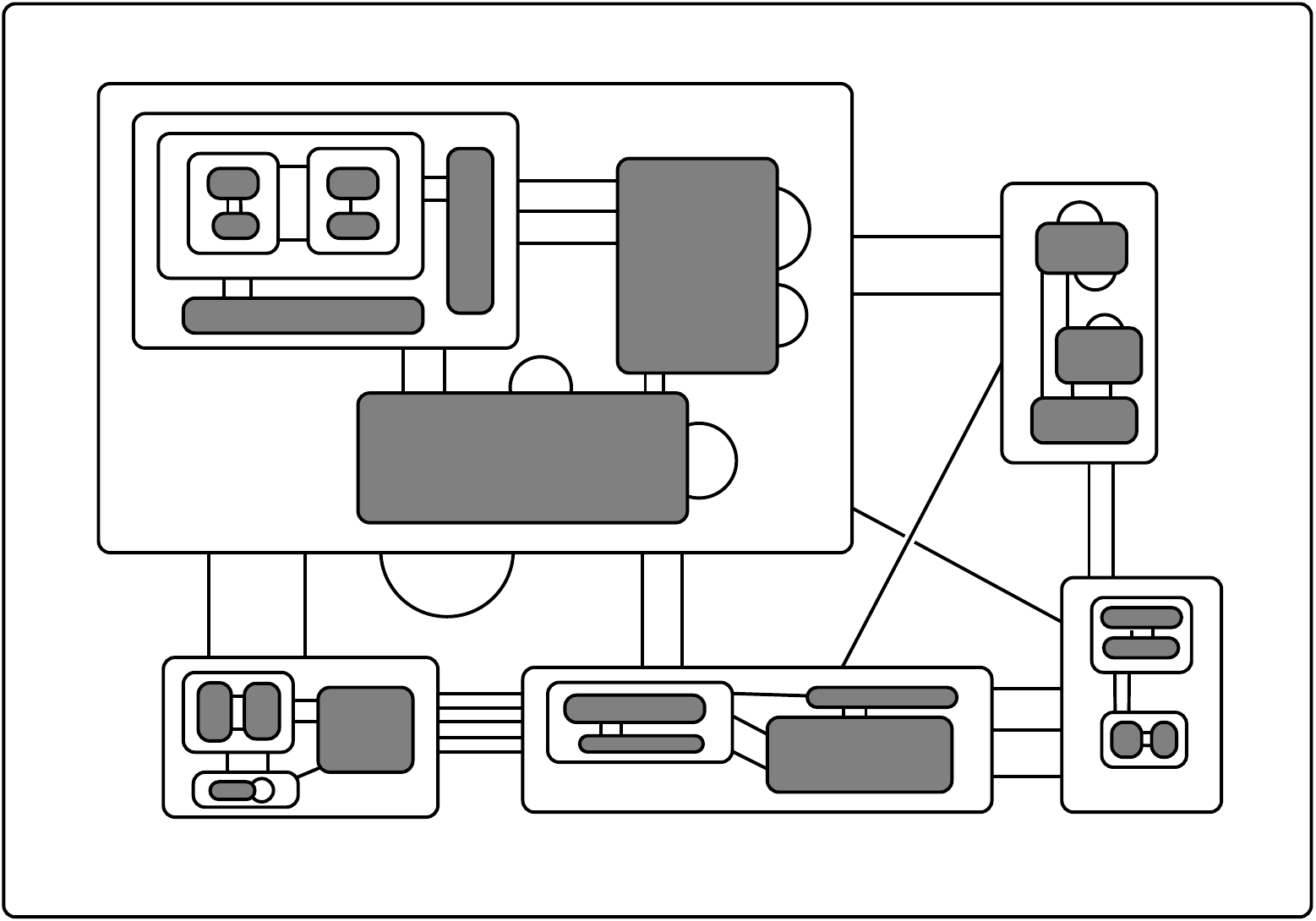}}
  \caption{Schematic picture of an $\HH$--structure. The outer box
    represents the top of the $\HH$--structure, and the nesting
    indicates the hierarchy. Lines connecting rounded boxes are
    edges of the graph of groups decomposition at that level. Shaded
    boxes are terminal groups in the structure, and are either slender
    or are equipped with an action of their associated groups on an
    $\HH$--complex.}
  \label{fig::hstructure}
\end{figure}

\subsection*{Resolving the action of $G$ on $T$}

In this section $\HH$ is assumed to be a slender hyperbolic hierarchy
of a finitely generated group. Let $X$ be a triangular complex with a
$G$ action and let $T_X$ be the cutpoint tree. Collapse edges with
non-small stabilizers to obtain $D_X$.

\begin{lemma}
  \label{lem::newhierarchy}
  Let $\mathcal{X}_G$ be an $\HH$--structure for $G\in\HH$. There are
  $\HH$--structures $\mathcal{X}_{G_v}$, $v\in\Delta_G$, such that
  \[\sum_{v\in\Delta_G}\covo(\mathcal{X}_{G_v})\leq \covo(\mathcal{X}_G)\] and
  if, for each terminal group $B$ of $\mathcal{X}_G$, $X_B$ is
  a point, then
  \[\height(\mathcal{X}_{G_v})\leq\height(\mathcal{X}_G)\]
\end{lemma}

\begin{proof}
  If there are no nondegenerate terminal vertex groups set
  $\mathcal{Y}_G=\mathcal{X}_G$. Each terminal vertex group in the
  resulting decomposition is either $\HH$--elliptic or slender. Note
  that $\height(\mathcal{Y}_G)\leq\height(\mathcal{X}_G)$.

  Let $\mathcal{X}_G$ be an $\HH$--structure on $G$, and suppose $G$
  acts on a slender $G$--tree $T$ with quotient $\Delta_G$. Suppose
  $\mathcal{X}_G$ has a nondegenerate terminal group $B$ acting
  on an $\HH$--complex $X_B$. Let $X'$ be the complex associated to
  $B$ and $T$ provided by Lemma~\ref{lem::collapsesmall}. Let
  $\Omega_B$ be the graphs of groups decomposition $D_{X'}/B$, and for
  each vertex $w$ of $\Omega_B$ let $X_w$ be the subcomplex of $X'$
  stabilized by $B_w$. There is a natural $B_w$--map $X_w\to\hat T$
  obtained by restriction, with $X_w^*$ connected. Let $Y_w$ be the
  $\HH$--complex $(X_w)_T$. Now let $\Omega_{B_w}=D_{Y_w}/B_w$, and
  for each vertex $z$ of $D_{Y_w}/B_w$ let $(B_{w})_z$ act on the
  subcomplex of $Y_w$ corresponding to a lift of $z$. Repeat over all
  nondegenerate terminal groups $B$ of $G$ to obtain an
  $\HH$--structure $\mathcal{Y}_G$.

  Let ${G_v}$ be a vertex group of $\Delta_G$, and let
  $\mathcal{X}_{G_v}$ be the hierarchy of ${G_v}$ provided by
  Lemma~\ref{lem::newhstructure} applied to $\mathcal{Y}_G$ and
  $\Delta_G$.  Let $W$ be a terminal vertex group of
  $\mathcal{Y}_G$. By construction $W$ is elliptic in $\Delta_G$ and
  $\Delta_W$ is a finite tree representing the trivial graph of groups
  decomposition of $W$.  Suppose first that $W$ is nondegenerate. For
  each vertex group $V$ of $\Delta_W$ in $\mathcal{X}_{G_v}$, if $V$
  is slender let $X_V$ be a point, and if $V=W<W$ let $X_V=X_W$, the
  $\HH$--complex associated to $V$. If $W$ is degenerate then each
  vertex group $W$ of $\Delta_W$ is either $\HH$--elliptic or slender
  and in these cases let $X_W$ be a point.
\end{proof}

\subsection*{Hierarchy of $\HH$--structures}

Let $H$ be a finitely generated group, $\mathcal{H}$ a slender
hyperbolic hierarchy for $H$, and suppose $H$ is $\HH$--almost
finitely presented. Let $\mathcal{X}_H$ be the trivial
$\HH$--structure with trivial graph of groups decomposition, and let
$X_H$ be any $\HH$--complex for $H$. 

Suppose that $\mathcal{X}_L$ has been defined for $L\in\mathcal{H}.$
For $\Delta_L$ and $Z$ a vertex group of $\Delta_L$, let
$\mathcal{X}_Z$ be the $\HH$--structure on $Z$ constructed in the
previous subsection. Let $B_{L,1},\dotsc,B_{L,n_L}$ be the terminal
vertex groups acting on nondegenerate $\HH$--complexes
$X_{B_{L,i}}$. Let $L_1,\dotsc,L_k$ be the descendants of $L$. Then
each $X_{B_{{L_j},k}}$ is obtained from some $X_{B_{L,i(k)}}$ by
resolving the action of $B_{L,i(k)}$ on $T_{\Delta_L}$. We call the
collection of $B_{L_j,i(k)}$ such that $i(k)=i$ the \emph{descendants}
of $B_{L,i}$. Since $X_H$ has finitely many triangles, for all but
finitely many $L$, each $B_{L,i}$ has exactly one descendant
$B_{{L_j},i}$ and $\covo(X_{B_{L,i}})=\covo(X_{B_{{L_j},i}})$. We
have:

\begin{lemma}[{\emph{c.f.}~\cite[p.~627]{dp}}]
  \label{lem::onedescendant}
  \[
  \covo(X_{B_{L,i}})\geq \sum_{k\mid i(k)=i} \covo(X_{B_{{L_j},k}})
  \] 
  and for all but finitely many $L\in\HH$, the sum on the right is
  over one element and the inequality is an equality. There is some
  $N_{\mathrm{tri}}$ so that for $i\geq N_{\mathrm{tri}}$ this is the case.
\end{lemma}

Henceforth $i\geq N_{\mathrm{tri}}$.

\section{Nondegenerate complexes converge to trees}
\label{sec::stabilizing}

The aim of this section is to replace, for groups sufficiently far
down the hierarchy $\HH$, each $\HH$--structure $\mathcal{X}_L$ with
an $\HH$--structure with no triangles. This, along with the fact
that in this case the depth of the
$\HH$--structures is non-increasing (Lemma~\ref{lem::newhierarchy}),
will complete the proof of Theorem~\ref{thm::maintheorem}.

Consider the finite collection of infinite sequences of terminal
vertex groups
\[\{G^p_{N_{\mathrm{tri}}}>G^p_{N_{\mathrm{tri}}+1}>G^p_{N_{\mathrm{tri}}+2}>\dotsc\}\]
such that $G^p_i\in\mathcal{X}_{L(p,i)}$, $L(p,i)\in\HH^i$, is
terminal, acts on a nondegenerate $\HH$--complex $X^p_{G^p_i}$, is the
only descendant of $G^p_{i-1}$, such that
$\covo(X^p_{G^p_i})=\covo(X^p_{G^p_{i+1}})$. To simplify notation we
drop the `$p$'s and denote $G^p_i$ by $G_i$ and $X^p_{G^p_i}$ by
$X_i$. See Figure~\ref{fig::theprocedure}. 

Note that if $v$ is a vertex in $X_i$ with non $\HH$-elliptic
stabilizer then the stabilizer of $v$ is slender, hence all
stabilizers of connected components of the link of $v$ are slender
and, following the steps in the construction of resolving complexes,
$v$ is not a cutpoint of $X_i$, hence the link $l$ of $v$ has exactly
one connected component and
$\stab(l)=\stab(v)$. %N{Modified this to be more consistent with
                    %changes in the previous section}
\begin{figure}
\labellist
\pinlabel $X_i{\xrightarrow{\rho_i}}{\hat T_i}$ [l] at 223 347
\pinlabel $X_i^*{\xrightarrow{\rho_i}}T_i$ [l] at 223 210
\pinlabel $\cup$ at 164 281
\pinlabel $\varphi_i$ [l] at 217 134
\pinlabel $(X_i)_{T_i}$ [tl] at 199 61
\pinlabel $X_{i+1}{\xrightarrow{\rho_{i+1}}}{\hat T_{i+1}}$ [l] at 596 347
\pinlabel $X_{i+1}^*{\xrightarrow{\rho_{i+1}}}T_{i+1}$ [l] at 596 210
\pinlabel $\cup$ at 537 281
\pinlabel $\varphi_{i+1}$ [l] at 593 134
\pinlabel $(X_{i+1})_{T_{i+1}}$ [tl] at 575 61
\endlabellist
\centerline{\includegraphics[width=.9\textwidth]{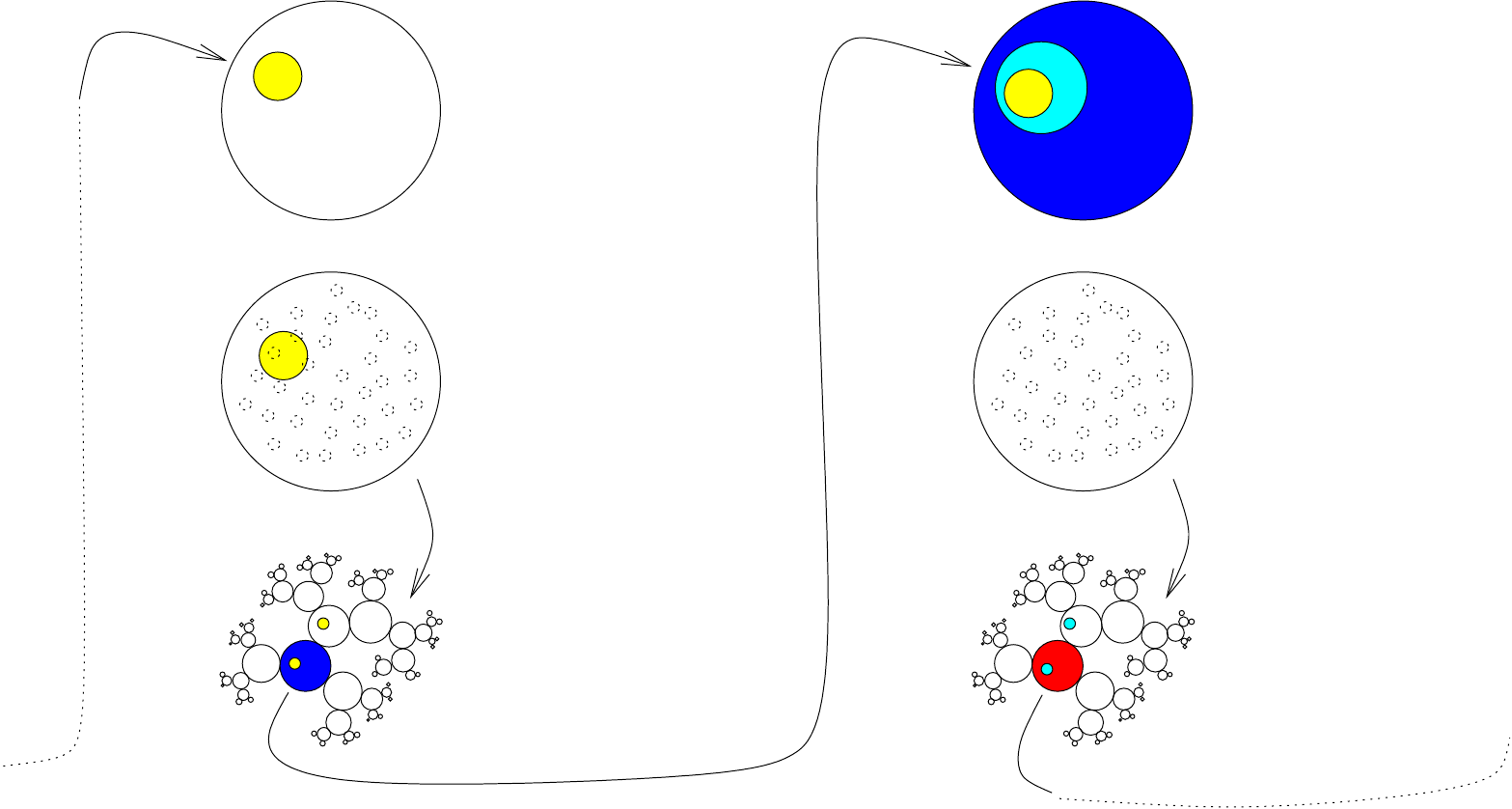}}
\caption{Sufficiently far down $\HH$ the descendant $X_{i+1}$ of $X_i$
  is a non $\HH$-elliptic cutpoint-free component of $(X_i)_{T_i}$
  constructed from a type 1 resolution.}
\label{fig::theprocedure}
\end{figure}
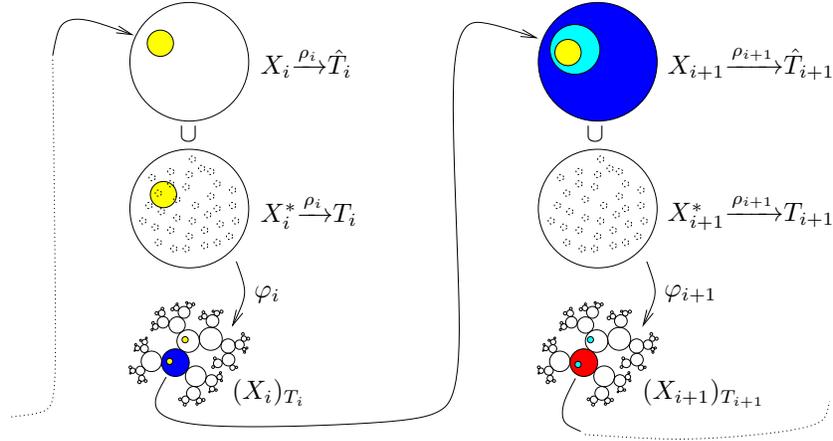

% \begin{lemma}[{\cite[Proposition 4.2]{dp}}]
%   \label{lem::hellipticcellstabs}
%   Edge and face (but not necessarily vertex!) stabilizers in $X_i$ are
%   $\HH$--elliptic.
% \end{lemma}

% \begin{proof}
%   Otherwise for some $i$, $\rho_i^{-1}(\partial T_i)$ contains a
%   $1$--cell, which when collapsed (Recall that the construction of
%   $X_{i+1}$ depends on whether $\rho_i$ is type 1 or type 2.),
%   decreases the number of orbits of triangles, hence
%   $\covo(X_i)>\covo(X_{i+1})$.
% \end{proof}

Let $\Delta_i$ be the decomposition $G_i$ inherits from
$\Delta_{L(p,i)}$, and let $T_i$ be the associated tree. We now argue
that for sufficiently large $i$, $X_i$ can be replaced by a graph of
groups such that each vertex group acts on a tree with $\HH$--elliptic
or slender vertex stabilizers.

Let $\mathcal{L}_i$ be the collection of orbits of connected
components of links of vertices of $X_i$, and denote the orbit of a
link $l$ by $\left[l\right]$. Say that
$\left[l\right]\in\mathcal{L}_i$ \emph{dies} in $X_{i+1}$ %N{AO}
if $\stab(l)$ acts hyperbolically in $T_i$, otherwise $\left[l\right]$
\emph{survives}. Let $\mathcal{L}^s_i$ be the collection of orbits of
connected components of links of vertices which survive, and let
$\mathcal{I}_i$ be the collection of orbits of links of vertices which
survive forever. There is a natural map
$\iota_i\colon\mathcal{I}_i\to\mathcal{I}_{i+1}$, and since
$\covo(X_{i+1})\leq\covo(X_{i})\leq\covo(X_H)$, eventually $\iota_i$
is bijective. The links which survive forever have $\HH$--elliptic
stabilizers, and if $\stab(l)$ is $\HH$--elliptic then $l$ is
\emph{$\HH$--elliptic}.

Let $\mathcal{V}_i$ be the collection of orbits of vertices $v$ such
that $v$ has an $\HH$--elliptic component in its link.  Since the
number of orbits of connected components of links which survive
forever is constant, $\vert\mathcal{V}_i\vert$ is non-decreasing in
$i$, and is eventually constant. Furthermore, if a component of the
link of $v$ is not slender then \emph{all} components of the link of
$v$ have non-slender stabilizer, and if the link of $v$ has an
$\HH$--elliptic component $l$ and $\stab(l)$ is slender then $l$ is
the only component of $\link(v)$.

Let $l$ be an $\HH$--elliptic component of a link. Then \[\vert
l/\stab(l)\vert\geq%N{AQ}
\vert l'/\stab(l')\vert,\,\,\,\,\,\, l'\in\iota_i(\left[l\right])\]
For sufficiently large $i$ this number stabilizes as well, giving a
bijection $\mathcal{V}_i\to\mathcal{V}_{i+1}$ and, for each vertex $v$
with an $\HH$--elliptic link component, a $\stab(v)$--equivariant
isomorphism of links $\link_{X_i}(v)\to\link_{X_{i+1}}(\varphi_i(v))$
($\varphi_i$ is shown in Figure \ref{fig::theprocedure}.)%N{AN}

We assume below that $N_{\mathrm{link}}\geq N_{\mathrm{tri}}$ has been
chosen large enough to arrange all of the above, over all sequences
$\{G^p_i\}$, for $i\geq N_{\mathrm{link}}$.

For $i\geq N_{\mathrm{link}}$, $X^*_i$ is connected and
$\varphi_i\colon X^*_i\to (X_i)_{T_i}$ induces bijections on orbits of
triangles and stars of vertices with $\HH$--elliptic components in
their links.

\subsection*{Finding $\HH$--elliptic subgroups}

Let $\triangles(X_i)$ be the set of triangles in $X_i$, and let
$\tau_i\colon\triangles(X_i)\to\triangles((X_{i})_{T_i})$ be the
induced map. A pair of triangles is an unordered pair of triangles
$(t,t')$ where $t,t'\in\triangles(X_i)$ overlap in an edge. Denote the
collection of orbits of pairs of triangles in $X_i$ by
$\tripair(X_i)$.

The map $\varphi_i$ \emph{separates} a class of pairs
$P=\left[(t,t')\right]\in\tripair(X_i)/G_i$ if $\tau_i(t)$ and
$\tau_i(t')$ lie in different cutpoint free components of
$(X_{i})_{T_i}$. See Figure~\ref{fig::separated}. If $\varphi_i$
doesn't separate $P$ then it descends to an element
\[
\tau_i(P)=\left[(\tau_i(t),\tau_i(t'))\right]\in\tripair(X_{i+1})
\] 
Similarly, $\varphi_{i+1}$ \emph{doesn't separate} $P$ if $\varphi_i$
doesn't separate $P$ and $\varphi_{i+1}$ doesn't separate
$\tau_i(P)$. Likewise for $\varphi_k$ for $k=i+2,\dotsc$.

\begin{definition}[Stable pairs of triangles]
  Let $\spairs(X_i)$ be the collection of equivalence classes of pairs
  of triangles which are not eventually separated by any $\varphi_j$,
  $j\geq i$. Elements of $\spairs(X_i)$ are \emph{stable pairs}.
\end{definition}

There are induced (injective) maps 
\[
\sigma_{i,j}\colon\spairs(X_i)\to\spairs(X_j)\] The purpose of this
section is to show that the sequence
\begin{equation}\label{eq::sigma-sequence} %N{Named it and will refer
                                %to it later}
\dotsb\to\spairs(X_i)\xrightarrow{\sigma_{i,i+1}}\spairs(X_{i+1})\to\dotsb
\end{equation}
eventually stabilizes.

\begin{figure}[ht]
\centerline{\includegraphics[width=.8\textwidth]{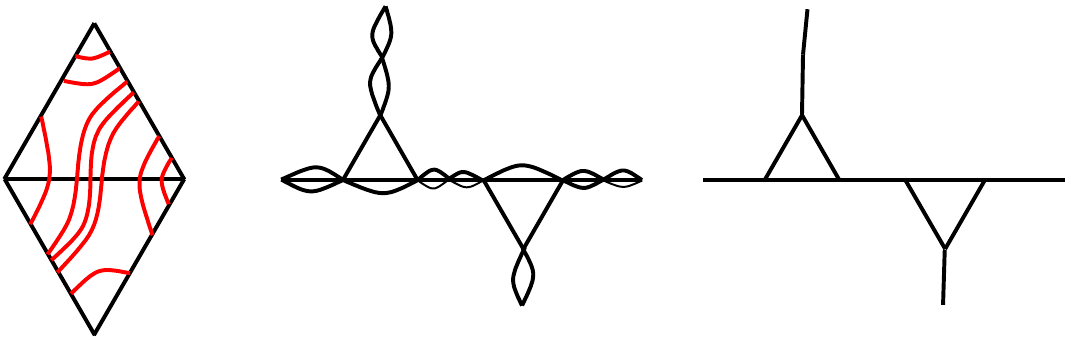}}
\caption{An adjacent pair of triangles in $X_i$ separated by
  $\varphi_i$.}
\label{fig::separated}
\end{figure}

% The
% next lemma is immediate.

% \begin{lemma}
%   \label{lem::connectingchains}
%   % If $\left[(s,t)\right],\left[(s,t')\right]\in\spairs(X_i)$ then
%   % $t\sim t'$. If
%   % \[
%   % \left[(t_0,t_1)\right],\left[(t_1,t_2)\right],\dotsc,\left[(t_{n-1},t_n)\right]\in\spairs(X_i)
%   % \]
%   % and $t_0$ and $t_n$ share an edge $e$ then
%   % $\left[(t_0,t_n)\right]\in\spairs(X_i)$.
%   If $\left[(s,t)\right],\left[(s,t')\right]\in\spairs(X_i)$ then
%   $t\sim_i t'$.
% \end{lemma}

Let $\sim_i$ be the equivalence relation on $\triangles(X_i)$
generated by $s\sim_i t$ if $\left[(s,t)\right]\in\spairs(X_i)$.  Let
$\{P_{\alpha}\}$ be the collection of subcomplexes, each of which is
the union of elements in a $\sim_i$ equivalence class, and let
$P_i^1,\dotsc,P^{n_i}_i$ be a set of representatives of orbits under
the action of $G_i$. Then $\cup_j(G_i/\stab(P_i^j))P_i^j$ contains all
triangles in $X_i$, and %by Lemma~\ref{lem::connectingchains}, 
if
$gP_i^j$ and $hP_i^{j'}$ overlap in a triangle then $j=j'$ and
$h^{-1}g\in\stab(P_i^{j})$.

Each $P_i^j\subset X_i$ pushes forward under $\varphi_i$ to a
subcomplex $\varphi_i(P_i^j)$ of $(X_{i+1})_{T_i}$ and there is an
element $h_{i,j}\in G_i$ such that $h_{i,j}\varphi_i(P_i^{j'})
\subset P_{i+1}^j$. Abusing notation, we will suppress mentioning the
elements $h_{i,j}$ and simply say that $P_i^j$ pushes forward to a
subcomplex of $P_{i+1}^{j'}$. Similarly for the stabilizers of the
$P_i^j$: $h_{i,j}\stab(P_i^j)h_{i,j}^{-1}\leq\stab(P_{i+1}^{j'})$, but
we will drop the $h_{i,j}$ and simply say that
$\stab(P_i^j)\leq\stab(P_{i+1}^{j'})$.%N{AW}

Since every triangle in $X_{i+1}$ is contained in some $P_{i+1}^{j'}$
and $n_i\geq n_{i+1}\geq 1$, we can assume from now on that $i$ is
chosen sufficiently large so that $n_i=n_{i+1}$, and that $P_i^j$
pushes forward to a subcomplex of $P_{i+1}^j$. Let $E_i^j$ be the
number of orbits of edges in $P_i^j$ under the action of
$\stab(P_i^j)$. Then $E_i^j\geq E_{i+1}^j$. Since the number of orbits
of edges is bounded from above by $3\covo(X_i)$ this quantity is
nonincreasing as well. Choose $N_{\mathrm{edges}}\geq
N_{\mathrm{link}}$ sufficiently large so that $E_i^j=E_{i+1}^j$ for
$i\geq N_{\mathrm{edges}}$.

\begin{lemma}
  \label{lem::increasestabilizer}
  If $\sigma_{i,i+1}$ in (\ref{eq::sigma-sequence}) is not
  bijective, %N{AX}
  $i>N_{\mathrm{edges}},$ then there is $j$ and an edge $e$ in $P_i^j$
  such that
  \[
  \stab^+_{P_i^j}(e)\lneq\stab^+_{P_{i+1}^j}(e)
  \]
  Furthermore $\stab^+_{P_i^j}(e)$ is conjugate into a
  non-$\HH$--elliptic edge group of $T_i/G_i$.
\end{lemma}
%N{Isn't this simply the case since the $\sim$--classes only increase? YES}

\begin{proof}
  If $\spairs(X_i)\into\spairs(X_{i+1})$ is not surjective there are
  triangles $t\subset P_i^j$ and $t'\subset gP_{i}^{j'}$ %N{Stuff not
                                %living in right place}
  and edges $e\subset t$, $g\cdot e\subset t'$ with $g\in
  G_i\setminus\stab(P_i^j)$ such that
  $\left[(t,t')\right]\not\in\spairs(X_i)$ but
  $\left[\tau_i(t),\tau_i(t')\right]\in\spairs(X_{i+1})$. Since $g$
  doesn't stabilize $P_i^j$, clearly
  $\stab^+_{P_i^j}(e)\lneq\stab^+_{P_{i+1}^j}(e) \ni g$.%N{oops}

  Since $t$ and $t'$ don't form a stable pair but their push-forwards
  do, there is a component $\lambda$ of $\Lambda$ that meets both $t$
  and $t'$ in the edges $e, g\cdot e$ respectively. Then
  $\stab^+_{P_i^j}(e)\leq\stab(\lambda)$ %N{Missing letters...}
  and $\stab(\lambda)$ is conjugate into an edge group of
  $\Delta_i$. Since $i\geq N_{\mathrm{link}}$ no component of the link
  of the vertex of $X_{i+1}$ corresponding to $\lambda$ is
  $\HH$--elliptic, otherwise a new equivalence class of
  $\HH$--elliptic link stabilizers would have to have appeared,
  contradicting the fact that the map
  $\iota_i\colon\mathcal{I}_i\to\mathcal{I}_{i+1}$ is a bijection for
  $i>N_{\mathrm{link}}$.
\end{proof}

Since $\HH$ satisfies the acc (Definition~\ref{def::acc}.) on
$\mathcal{C}_{\HH}$, there is some first
index $M_p\geq N_{\mathrm{edges}}$ (recall we are working in the
branch $G^p_i=G_i$) such that for every
edge $e \subset P^j_i$, $\stab^+_{P_i^j}(e)=\stab^+_{P_{i+1}^j}(e)$
 for $i\geq M_p$, hence $\sim_i=\sim_{i+1}$ for $i\geq M_p$
by Lemma~\ref{lem::increasestabilizer}.

\subsection*{Proof of Theorem~\ref{thm::maintheorem}}

\begin{definition}
  \label{def::unstableedge}
  An \emph{unstable} edge is an edge $e$ such that there are triangles
  $t,t'$ with $t\cap t'=e$ but
  $\left[(t,t')\right]\not\in\spairs(X_i)$. Let $W\subset X_i$ be the
  union of unstable edges in $X_i$.
\end{definition}

A \emph{cone} $C$ is a triangulated disk with exactly one interior
vertex. A cone in a triangular complex $X$ is a combinatorial map
$\gamma\colon C\to X$ which maps triangles to triangles. A cone
$\gamma\colon C\to X$ is \emph{simple} if the associated path in the
link of the image of the cone point is simple.

Let $C\to X$ be a cone in $X$ and let $C^*$ be the space obtained by
removing vertices of $C$ which are mapped to $\rho^{-1}(\partial
T)$. Let $\Lambda$ denote also the preimage of $\Lambda$ in $C^*$. The
map $C^*\to X^*/\Lambda$ induces maps $C^*/\Lambda\to X^*/\Lambda =
X_T$, %N{AZ}
where $C*/\Lambda$ is the space obtained by collapsing each connected
component of $\Lambda$ in $C$ to a point, followed by collapsing
bigons to edges, i.e. reducing.
Let $c$ be the cone point in $C$, and let $s$ be the outermost
component of $\Lambda$ encircling $c$, if there is one, otherwise let
$s=c$. The \emph{push-forward} $C'$ of $C$ to $X_T$ is the cone
obtained from $C^*/\Lambda$ by taking all triangles in $C^*/\Lambda$
containing the image of $s$. See Figure~\ref{fig::conesfromcones}.

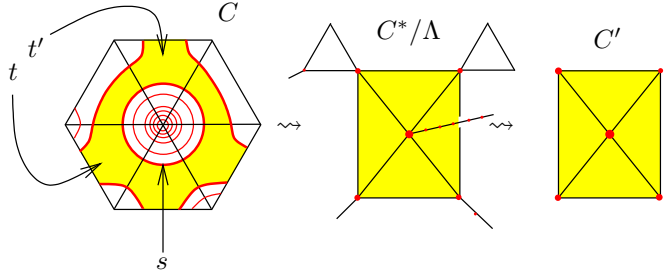
\begin{figure}[ht]
\labellist
\pinlabel $C$ [bl] at 76 86 
\pinlabel $C^*/\Lambda$ [b] at 150 76
\pinlabel $C'$ [b] at 225 75
\pinlabel $s$ [t] at 57 2
\pinlabel $t$ [b] at 1 63
\pinlabel $t'$ [tr] at 15 84
\pinlabel $\leadsto$ at 105 48
\pinlabel $\leadsto$ at 185 48
\endlabellist
\centerline{\includegraphics{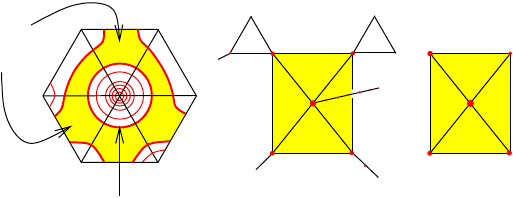}}
\caption{Constructing the push-forward $C'$ of $C$. In this example
  the triangles $t$ and $t'$ are not adjacent, but have adjacent
  push-forwards.}
\label{fig::conesfromcones}
\end{figure}

\begin{lemma}
  \label{lem::conelemma}
  Suppose $C_i\to X_i$ is a simple cone, and that there are two %{Any two}
  %adjacent 
  triangles $t$ and $t'$ in the image of $C_i$ such that $t \not\sim_i
  t'$, then $i<M_p$.
\end{lemma}

\begin{proof}
  Suppose $i\geq M_p$. Let $C_j\to X_j$ be the push-forward of $C_i$
  to $X_j$. Since there are triangles $t$ and $t'$ such that
  $t\not\sim_i t'$ then for some $j>i$, $\vert C_j\vert <\vert
  C_i\vert$, otherwise each pair of adjacent triangles in $C_i$ is a
  stable pair. Let $j$ be the first index such that $\vert
  C_j\vert=\vert C_{j'}\vert$ for $j'\geq j$. Then all triangles in
  the image of $C_j$ are $\sim_j$ equivalent. Let $t_1,\dotsc,t_n$ be
  the triangles in the image of $C_j$, indexed so that
  $\left[(t_k,t_{k+1})\right]\in\spairs(X_j)/G_j$. Let $\tilde t_k$ be
  the triangle in the image of $C_i$ in $X_i$ corresponding to
  $t_k$. Then since $i\geq M_p$ there are edges $\tilde e_k$ in $X_i$
  such that $\tilde t_k\cap \tilde t_{k+1}=\tilde e_k$, hence
  $\left[(\tilde t_k,\tilde t_{k+1})\right]\in\spairs(X_i)/G_i$, but
  this implies that the cone $C_i\to X_i$ was not simple.
\end{proof}

\begin{lemma}
  \label{lem::separatinglemma}
  Suppose $i>M_p$, that $t,t'\in\triangles(X_i)$ intersect in an edge
  $e$, and that $\left[(t,t')\right]\not\in\spairs(X_i)$. Then $e$
  separates $X_i$, with $t\setminus e$ and $t'\setminus e$ lying in
  different components of $X_i\setminus e$.
\end{lemma}

Note that there may be edges which are not unstable, but which
still separate $X_i$.

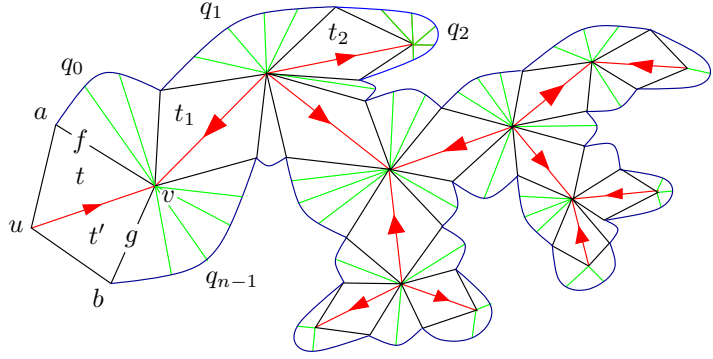
\begin{figure}[ht]
\labellist
\pinlabel $t$ at 21 73
\pinlabel $t'$ at 28 49
\pinlabel $t_1$ at 64 98
\pinlabel $t_2$ at 127 131
\pinlabel $q_0$ [br] at 25 112
\pinlabel $q_1$ [br] at 82 135
\pinlabel $q_2$ [l] at 168 132
\pinlabel $q_{n-1}$ [tl] at 70 36
\pinlabel $a$ [br] at 10 94
\pinlabel $b$ [tr] at 33 29
\pinlabel $u$ [r] at 1 52
\pinlabel $v$ [tl] at 52 69
\pinlabel $f$ at 21 87
\pinlabel $g$ at 42 47
\endlabellist
\centerline{\includegraphics[width=.75\textwidth]{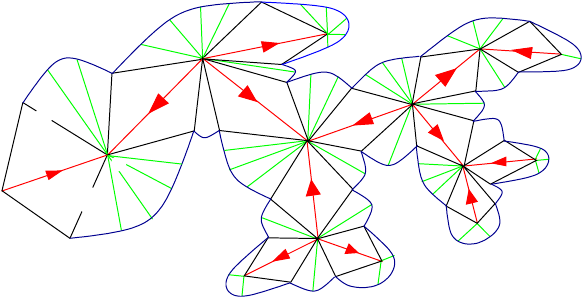}}
\caption{The homology $h$ from Lemma~\ref{lem::separatinglemma}, which
  we may assume is a disk. The edges with arrows are mapped to $e$.}
\label{fig::homotopyfigure}
\end{figure}

\begin{proof}[Proof of Lemma~\ref{lem::separatinglemma}.]
  Let $a$ and $b$ be the vertices of $t$ and $t'$, respectively,
  distinct from the endpoints $v$ and $w$ of $e$. Suppose that $e$
  doesn't separate $X_i$ into at least two components, with
  $t\setminus e$ lying in one and $t'\setminus e$ lying in
  another. Then there is an edge path $q\colon I\to X_i$ of a
  subdivided interval such that $q(0)=a$ and $q(1)=b$ such that
  $q^{-1}(e)=\emptyset$. Let $f$ and $g$ be the oriented edges of $t$
  and $t'$ connecting $a$ to $v$ and $v$ to $b$, respectively. Let
  $h'\colon D'\to X_i$ be a combinatorial map of a triangulated
  surface $D'$ representing a homology between the edge paths $gf$ and
  $q$, and let $h\colon D\to X_i$ be the combinatorial map of a
  surface obtained by attaching two triangles representing $t\cup_e
  t'$ to $D'$. See Figure~\ref{fig::homotopyfigure}.

  Without loss, by perhaps changing $q$ and $h$, we may assume that
  the union of edges of $D$ which are mapped to $e$ does not separate
  $D$, and that $D$ is a disk, as illustrated in
  Figure~\ref{fig::homotopyfigure}. The path $q$ may be divided into
  subpaths $q_0,\dotsc,q_{n-1}$ such that $q_j$ connects the apex of a
  triangle $t_j$ to the apex of a triangle $t_{j+1}$, and such that
  the side of $t_j$ is mapped to $e$ by $h$. Furthermore, by
  identifying the edges labeled $e$ in the sequence of triangles
  determined by $t_j$, $q_j$, and $t_{j+1}$, we obtain a cone in the
  link of one of $v$ or $w$. Then either $t_j=t_{j+1}$ or there is a
  simple cone $C\to X_i$ containing $t_j$ and $t_{j+1}$, hence by
  Lemma~\ref{lem::conelemma} $t_j\sim_it_{j+1}$ for all $j$, therefore
  $t\sim_i t'$, contrary to hypothesis.
\end{proof}

% \begin{lemma}
%   \label{lem::neitherendpointishelliptic}
%   Suppose $i>N_{\mathrm{link}}$, and that
%   $\left[(t,t',e)\right]\not\in\spairs(X_i)$. Then neither stabilizer
%   of a link of a vertex corresponding to $e$ is $\HH$--elliptic.
% \end{lemma}

% \begin{proof}
%   Suppose not. Let $v$ and $w$ be the vertices joined by $e$, such
%   that the component $l$ of the link of $v$ corresponding to $e$ is
%   $\HH$--elliptic. Then the link of $v$ in $X_{j}$, for some $j>i$,
%   has strictly fewer orbits of edges than $l$, contradicting our
%   choice of $N_{\mathrm{link}}$.
% \end{proof}

% Let $P_i^1,\dotsc,P^{n_i}_i$ be a
% collection of representatives of orbits of unions of $\sim_i$
% equivalence classes in $X_i$. Then $gP_i^k$ and $hP_i^l$ either
% coincide, intersect in a cut-edge $e$, a cutvertex $w$, or are
% disjoint.

% \subsection*{Replacing $X_i$ by a tree of planes}

\begin{proof}[Proof of Theorem~\ref{thm::maintheorem}]
  Fix some $N_{\mathrm{equiv}}>\max_p\{M_p\}$, and let
  $\{U^p_{\alpha}\}$ be the collection of maximal connected
  subcomplexes of $X^p_{N_{\mathrm{equiv}}}$ which aren't eventually
  separated by any $\varphi^p_j$, $j>{N_{\mathrm{equiv}}}$. Each
  $U^p_{\alpha}$ is a union of $\sim_{N_{\mathrm{equiv}}}$ equivalence
  classes which are either disjoint or meet in a vertex with
  nonslender $\HH$--elliptic link component stabilizers. Clearly
  $\stab(U^p_{\alpha})$ is $\HH$--elliptic. Let $T^p$ be the bipartite
  graph whose vertex set is the collection of $U^p_{\alpha}$ and
  unstable edges, and whose edge set is the set of pairs
  $(U^p_{\alpha},e)$, where $e\subset U^p_{\alpha}$. Then clearly
  $T^p$ is connected, and since the endpoints of unstable edges are
  not cutpoints, Lemma~\ref{lem::separatinglemma} implies that $T^p$
  is a tree. Vertex stabilizers correspond to stabilizers of
  $U^p_{\alpha}$ and unstable edges, hence are $\HH$--elliptic or
  slender, and edge stabilizers are stabilizers of pairs
  $(U^p_{\alpha},e)$, hence are slender.  For each $p$, replace the
  $G^p_{{N_{\mathrm{equiv}}}}$--complex $X^p_{N_{\mathrm{equiv}}}$ by
  the graph of groups decomposition $T^p/G_{N_{\mathrm{equiv}}}$ given
  above.
\end{proof}

\section{Strong accessibility}
\label{sec::jsjhierarchy}

\subsection*{Almost finitely presented groups}

As it is rather long and technical, we will not restate the definition
of the {\jsj} decomposition of a finitely presented group over slender
edge groups here, and instead refer the reader to
\cite[Theorem~5.13]{fuji::jsj} and \cite{dunwoodyjsj}. We
need the following from~\cite{fuji::jsj}.

\begin{theorem}%N{BF}
  [{\cite[Theorem~5.15]{fuji::jsj}}]
  \label{jsjdecomposition}
  Let $G$ be a finitely presented group, and $\Gamma$ a graph
  decomposition we obtain in~\cite[Theorem~5.13]{fuji::jsj}
  \emph{[Note: $\Gamma$ is the slender \jsj.]}. Let $G=A*_CB$, $A*_C$
  be a splitting along a slender group $C$, and $T_C$ its Bass-Serre
  tree.
  \begin{enumerate}
  \item If the group $C$ is elliptic with respect to any minimal
    splitting of $G$ along a slender group, then all vertex groups of
    $\Gamma$ are elliptic on $T_C$.
  \item Suppose the group $C$ is hyperbolic with respect to some
    minimal splitting of $G$ along a slender group. Then
    \begin{enumerate}
    \item All non-enclosing vertex groups of $\Gamma$ are elliptic on
      $T_C$.
    \item For each enclosing vertex group, $V$, of $\Gamma$, there is
      a graph of groups decomposition of $V$, $\mathcal{V}$, whose edge groups
      are in conjugates of $C$, which we can substitute for $V$ in
      $\Gamma$ such that if we substitute for all enclosing vertex
      groups of $\Gamma$ then all vertex groups of the resulting
      refinement of $\Gamma$ are elliptic on $T_C$.
    \end{enumerate}
\end{enumerate}
\end{theorem}

In other words, the non-enclosing vertex groups of the slender JSJ
decomposition $\Gamma$ are elliptic in every slender splitting of $G$.
A few remarks are in order:

\begin{itemize}
\item Let $T_{C_1},\dotsc,T_{C_n},\dotsc$ be a collection of
  Bass-Serre trees associated to splittings of a \emph{finitely
    generated} group $G$ along slender edge groups. Then there is a
  graph of groups decomposition $\Gamma_n$ satisfying the bullets of
  Theorem~\ref{jsjdecomposition} for the trees
  $T_{C_1},\dotsc,T_{C_n}$. The decomposition $\Gamma_m$ is a
  refinement of $\Gamma_{m+1}$.
\item If $E_1,\dotsc,E_k<G$ is a family of subgroups such that each
  $E_i$ acts elliptically in $T_{C_j}$ for all $j$ then we may assume that
  each $E_i$ is elliptic in $\Gamma_m$ for all $m$.
\item If $G$ is accessible relative to $\{E_i\}$ over the family of
  slender subgroups, i.e., there is a constant bounding the number of
  vertices in a reduced graph of groups decomposition relative to
  $\{E_i\}$ of $G$ over slender edge groups, then there is a
  \emph{slender {\jsj}  decomposition} of $G$ relative to $\{E_i\}$, i.e.
  a graph of groups decomposition $\Gamma$ satisfying the conclusion
  of Theorem~\ref{jsjdecomposition}, where $T_{C}$ is only allowed to
  vary over all slender $G$--trees in which the $E_i$ are elliptic.
\end{itemize}

Let $H$ be almost finitely presented relative to $\mathcal{E}$, and
suppose that $H$ doesn't contain any slender subgroups outside
$\mathcal{E}$ which have an infinite dihedral quotient. Let $\Delta_H$
be the slender {\jsj}  decomposition of $H$ and let $X_H$ be an acyclic
simplicial complex that $H$ acts on with cell stabilizers which are
either slender or in $\mathcal{E}$. Let $\mathcal{X}_H$ be the trivial
hierarchy of $H$, where $\Omega_H$ is just a point. For each
non-slender vertex group $L$ of $\Delta_H$ let $\mathcal{X}_L$ be the
hierarchy obtained by resolving the action of $G$ on the tree
associated to $\Delta_H$. Then, by the above, $L$ is accessible
relative to $\mathcal{E}$, hence has a slender {\jsj}  decomposition
relative to $\mathcal{E}$. Repeat to construct a hierarchy $\HH$ of
$G$. We call $\HH$ the slender {\jsj}  hierarchy of $H$ relative to
$\mathcal{E}$. By construction $\mathcal{X}_L$ is an $\HH$--structure
for $L$. Corollary~\ref{cor::nodihedral} claims that $\HH$ is finite:

\begin{proof}[Proof of Corollary~\ref{cor::nodihedral}]
  Let $L\in\HH^n$ for some $n>N$, where $N$ is as in
  Theorem~\ref{thm::maintheorem}. Groups in $\HH$ are either
  slender-by-orbifold, hence have {\jsj} decompositions which are
  graphs of slender groups over slender edge groups, or are elliptic
  in the top level of the $\HH$--structure of their parents, hence if
  $K<K'$ is a vertex group of the {\jsj} of $K'$ and is not a graph of
  slender groups then we may assume $\height(\XX_K)<\height(\XX_K')$,
  hence $\HH_L$ is finite and has terminal leaves which are either
  slender or are the non-slender terminal leaves of $\XX_L$.
\end{proof}

\subsection*{Relatively hyperbolic groups}

In this section we prove Corollary~\ref{cor::relhyp}. Since relatively
hyperbolic groups are finitely presented relative to their peripheral
subgroups it suffices to show that relatively hyperbolic groups
satisfy the acc on $\mathcal{C}_{\HH}$.

\begin{lemma}
  \label{lem::relhypacc}
  Let $G$ and $\HH$ be as in Corollary~\ref{cor::relhyp}. Then $G$
  satisfies the acc on $\mathcal{C}_{\HH}$.
\end{lemma}

We have chosen to use a definition of relative hyperbolicity (first
introduced in \cite{Gromov-Essays}) which will facilitate the proof of
Lemma~\ref{lem::uniformlyboundedorders}: it is easily seen to be
equivalent to the standard definitions. See, for instance,
\cite[Definition 3.3]{Hruska}. If $Z$ is a $\delta$-hyperbolic metric
space, then it has a Gromov boundary $\partial Z$. Horofunctions are
defined in \cite[\S 2]{Hruska}; if $h:Z \to \mathbb{R}$ is a
horofunction centered at some $\xi \in \partial Z$, we denote by
$B(n)=\{x \in X \mid h(x)\geq n\}$ the \emph{depth-$n$ horoball.}

\begin{definition}[{See \cite[{Definition~3.3}]{Hruska}}]
  \label{defn::rel-hyp} 
  A group $G$ is \emph{hyperbolic relative to peripheral subgroups
    $P_1,\ldots, P_r$} if it acts properly on a $\delta$-hyperbolic
  graph $Z$ such that each peripheral subgroup $P_i$ fixes a point
  $p_i \in \partial Z$ and centered at each $p_i$ there is a
  horofunction $h_i$ so that if $B_i(0)$ is corresponding depth
  $0$-horoball, the $G$-translates of the $B_i(0)$ are all disjoint
  and $G$ acts cocompactly on $Z \setminus U$, where $U$ is the union
  of the translates of these horoballs. The points in $\partial Z$
  that are translates of the $p_i$ are called \emph{parabolic limit
    points}. We denote this set by $\Pi \subset \partial Z.$
\end{definition}

%%%%%%%%%%%%%%%%%%%%%%%%%%%%%%%%%%%%%%%%%%%%%%%%

\begin{lemma}
  \label{lem::uniformlyboundedorders}
  Let $G$ be a relatively hyperbolic group and let $Q<G$ be a finite
  subgroup. There is a constant $K$ such that if
  \begin{itemize}
  \item $H$ is a two-ended non-peripheral subgroup of $G$ and $Q<H$, or
  \item $Q$ is contained in two distinct conjugates of peripheral
    subgroups of $G$
  \end{itemize}
  then $\vert Q\vert \leq K$
\end{lemma}

The next lemma (needed for Lemma~\ref{lem::uniformlyboundedorders})
though not explicitly stated in \cite{Bogo-Ger}, is easily extracted
from their proof that there are only finitely many conjugacy classes
of finite subgroups of a hyperbolic group.

\begin{lemma}[(c.f. \cite{Bogo-Ger})]
  \label{lem::min-displ} 
  Let $Q$ be a finite group of isometries of a $\delta$-hyperbolic
  metric space $Z$ then there is a point $x_Q \in Z$ which is
  displaced by at most $3\delta$ by each element of $Q$.
\end{lemma}

\begin{proof}[Proof of
  Lemma~\ref{lem::uniformlyboundedorders}.]%What's wrong with a Lemmma
  Recall that every two-ended group $H$ either has an infinite
  dihedral quotient or splits as a semi-direct product
  \begin{equation}\label{eq::semidirect} H \approx
    Q\rtimes \group{t}
  \end{equation} 
  
  Since $G$ is relatively hyperbolic it acts freely (but not
  necessarily cocompactly) on a proper $\delta$--hyperbolic graph $Z$
  (c.f. \cite{GrM_DF}) such that the stabilizers of parabolic limit
  points $p$ in the Gromov boundary $\Pi$ of $Z$ are precisely the
  peripheral subgroups of $G$. Furthermore, there is a
  $G$--equivariant collection of disjoint horoballs $B(p)$ centered at
  parabolic limit points $p \in \Pi$ whose stabilizers $G_p$ are the
  peripheral subgroups of $G$; $G$ permutes this collection and maps
  horospheres to horospheres of the same depth. See \cite{Hruska}.

  Recall that Bowditch's characterization of relatively hyperbolic
  groups as those groups that act cocompactly on \emph{fine}
  hyperbolic graphs (see \cite[Definition 2]{bowditch2012relatively}
  or \cite[\S 3.3]{Hruska}) immediately implies that the intersections
  of any two distinct conjugates of peripheral subgroups of $G$ have
  orders bounded by some constant $K_1=K_1(G)$. Otherwise one could
  construct arbitrarily many circuits of some bounded length.

  By Lemma~\ref{lem::min-displ}, if $F$ is a finite group of isometries of a
  $\delta$-hyperbolic metric space $Z$ there is a point $x_F$ which is
  displaced at most $3\delta$ by each element of $F$. Let $Q$ and $t$
  be as in (\ref{eq::semidirect}) and suppose first that $x_Q$ is at
  least $3\delta$--deep in a horoball $B(p)$. Then $Q\cdot x_Q \subset
  B(p)$, hence $Q\cdot B(p)=B(p)$ and $Q \leq G_p$. On the other hand,
  $t Q t^{-1}$ must fix the point $tp \in \Pi$. Since $H$ is not
  parabolic $tp \neq p$, but since $\groupone{t}$ normalizes $Q$ we
  must have $Q \leq G_p \cap tG_pt^{-1}$, hence $\vert Q\vert \leq
  K_1$.

  Otherwise $x_Q$ lies in the neutered space $W=Z\setminus B(3\delta)$
  obtained by removing all $3\delta$-deep horoballs.  Since $G$ acts
  freely and cocompactly on $W$, the number of vertices in a ball of
  radius $3\delta$ in $W$ is bounded by some $K_2=K_2(G,W)$; thus
  $\vert Q\vert = \vert Q\cdot x_Q\vert \leq K_2$. Set
  $K=\max\{K_1,K_2\}$.
\end{proof}

\begin{figure}[ht]
\centerline{
  \includegraphics[width=.5\textwidth]{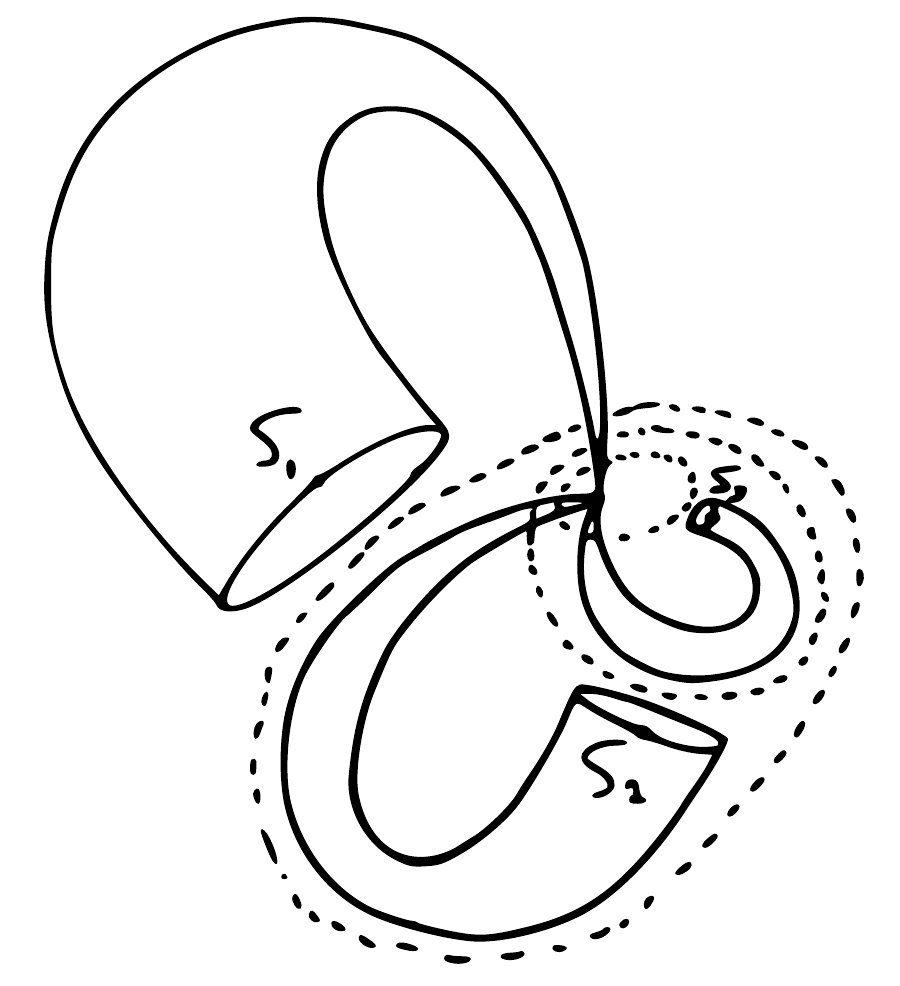}
}
\caption{A chain $S_i$ of $\HH$--elliptic finite subgroups of
  non-$\HH$--elliptic edge groups. In this case a peripheral subgroup
  is represented by the dot at the center of the picture. Each $S_i$
  is peripheral, but contained in two distinct conjugates of the
  peripheral subgroup, hence have orders uniformly bounded above.}
\label{fig::cornucopia}
\end{figure}

We are now finally prepared to prove Lemma~\ref{lem::relhypacc}.

\begin{proof}[Proof of Lemma~\ref{lem::relhypacc}]
  Let $K$ be the constant from Lemma~\ref{lem::uniformlyboundedorders}
  If $S_1<S_2<\dotsb$ is an ascending chain in $\mathcal{C}_{\HH}$,
  $S_i<H_i$, where $H_i$ is a non-$\HH$--elliptic edge group. If $H_i$
  is non-peripheral in $G$ then $\vert S_i\vert <K$. If $H_i$ is
  peripheral in $G$ then $H_i$ is two ended and hyperbolic in the
  slender {\jsj} decomposition of some lower level $L$ in $\HH$
  % (otherwise the number of $\HH$--elliptic vertex stabilizers strictly
  % increases, contradicting our choice of $N_{\mathrm{link}}$ )
  and $S_i$ is contained in an edge group $E$ of $\Delta_L$ (since
  $S_i$ is finite and $H_i$ fixes an axis.) Either $E$ is peripheral,
  in which case $S_i$ is contained in two distinct conjugates of
  peripheral subgroups of $G$, or $E$ is not peripheral. In either
  case $\vert Q\vert\leq K$. Since the $S_i$ are of uniformly bounded
  orders $G$ satisfies the acc on $\mathcal{C}_{\HH}$.
\end{proof}

%%%%%%%%%%%%%%%%%%%%%%%%%%%%%%%%%%%%

The following may also be of interest.

\begin{corollary}
  Let $G$ be a relatively hyperbolic group. There are only finitely
  many isomorphism classes of non-peripheral two-ended subgroups.
\end{corollary}

\begin{proof}
  If $H\leq G$ is a non-peripheral two-ended subgroup, then if it maps
  surjectively onto $\mathbb{Z}$, then it splits as in
  (\ref{eq::semidirect}), and the bound on the order of $Q$ given by
  Lemma \ref{lem::uniformlyboundedorders} bounds the number of
  isomorphism classes.  Otherwise $G$ surjects onto
  $\mathbb{Z}_2*\mathbb{Z}_2$ and therefore has an index 2 subgroup of
  the form (\ref{eq::semidirect}).
\end{proof}

\subsection*{Hierarchies over elementary families}

In this section we prove Theorem~\ref{dp::variation}.
The proof is formally identical to the proof of
Theorem~\ref{thm::maintheorem}, however, since we are allowed to
construct a hierarchy by hand, we don't need to use $\HH$--structures.

\begin{proof}{Proof of Theorem~\ref{dp::variation}}
  We define a hierarchy $\HH$ of $G$ inductively.

  {\bf Case 1:} Let $X$ be a $G$--complex with stabilizers in an
  elementary family $\mathcal{C}$. By collapsing cells with infinite
  stabilizers in $\mathcal{C}$ we may assume that edge and face
  stabilizers in $X$ are finite. Let $T_X$ be the cutpoint tree. Edge
  stabilizers in $T_X$ are either in $\mathcal{C}$ or are finite. If
  $G$ acts on $T_X$ with global fixed point there is a cutpoint free
  component $Y$ of $X$ stabilized by $G$. Go to case 2. Otherwise, let
  the descendants of $G$ be the vertex groups of $T_X/G$. The vertex
  groups of $T_X/G$ are either in $\mathcal{C}$, are finite, or are
  stabilizers of one of $Y_1,\dotsc,Y_n$, where each $Y_i$ is a
  representative of an orbit of cutpoint free components of $X$. Note
  that $\sum_i\covo(Y_i)=\covo(X)$.

  {\bf Case 2:} Suppose $G$ and $X$ are as above, and that $X$ has no
  cutpoints. Suppose that $X$ has a separating edge. Let $S_X$ be the
  cut-edge tree. If $G$ acts on $S_X$ with global fixed point there is a
  maximal connected subcomplex $Y$ of $X$ which is stabilized by $G$
  and doesn't have a separating edge. Replace $X$ by $Y$. Go to case
  3. If $G$ doesn't have a global fixed point, let the descendants of
  $G$ be the vertex groups of $S_X/G$. They are either finite or
  conjugate to a stabilizer of one of $Y_1,\dotsc,Y_n$, where each
  $Y_i$ is a representative of a maximal cut-edge-free component of
  $X$. Note that $\sum_i\covo(Y_i)=\covo(X)$.

  {\bf Case 3:} In the remaining case, $G$ has an action on a cutpoint
  and cut-edge free $G$--complex $X$. Suppose that $G$ has a
  nontrivial graph of groups decomposition over elements of
  $\mathcal{C}$. Then $G$ has a nontrivial graph of groups
  decomposition over elements of $\mathcal{C}$ in which every element
  of $\mathcal{C}$ acts parabolically, elliptically, or
  hyperbolically~\cite[Lemma~1.4]{dp}. Let $T$ be the associated tree and
  let $\rho\colon X\to \hat T$ be the resolving map. Then $X^*$ is
  connected and $G$ acts on $X_T$. Now we are again in the first case
  and $G$ doesn't act on $T_{X_T}$ with global fixed point.

  Since $\sum_i\covo(Y_i)\leq\covo(X)$, there are at most finitely many
  infinite branches
  \[
  G^p_1>G^p_2>\dotsb
  \]
  in $\HH$, where each $G^p_i$ is the sole nonelementary descendant of
  $G^p_{i-1}$. (This is the same principle as in
  Lemma~\ref{lem::onedescendant}.) As before, we drop the $p$ and let
  $X_i$ be the $G_i$ complex produced above. Again, there exist
  $N_{\mathrm{link}}\leq N_{\mathrm{edges}}$, so that $\HH$--elliptic
  vertex stabilizers in $X_i$ stabilize, and the number of orbits of
  edges in $\sim_i$ equivalence classes stabilizes. The ascending
  chain condition on finite subgroups elements of $\mathcal{C}$
  immediately implies the analogue of
  Lemma~\ref{lem::increasestabilizer}. Now argue, as in the proof of
  Theorem~\ref{thm::maintheorem}, that for some $N_{\mathrm{equiv}}$,
  $G_{N_{\mathrm{equiv}}}$ acts on a tree $T_p$ with $\HH$--elliptic
  or finite edge stabilizers.
\end{proof}

\bibliographystyle{amsalpha}
\bibliography{accbibnomr}

\bigskip

\begin{flushleft}
% Department of Mathematics\\
% University of Michigan \\
% Ann Arbor, MI 48109-1043\\
% USA\\
\emph{email:} \texttt{lars@d503.net}
\end{flushleft}

\begin{flushleft}

  \emph{email:} \texttt{nicholas.touikan@gmail.com}
\end{flushleft}

\end{document}